\documentclass[reqno,12pt,a4paper]{amsart}
\usepackage[T1]{fontenc}
\usepackage[utf8]{inputenc}
\usepackage[english]{babel}
\usepackage{amsmath,amssymb,amsthm}
\usepackage{bm}
\usepackage[margin=25mm]{geometry}
\usepackage[hidelinks]{hyperref}
\newcommand{\cN}{\mathbb N}
\newcommand{\cZ}{\mathbb Z}
\newcommand{\PP}{\mathbb P}
\newcommand{\cA}{\mathcal A}
\newcommand{\cB}{\mathcal B}

\theoremstyle{plain}
\newtheorem{theorem}{Theorem}[section]
\newtheorem{lemma}{Lemma}[section]
\newtheorem{corollary}{Corollary}[section]
\theoremstyle{definition}

\theoremstyle{remark}
\newtheorem{remark}{Remark}
\numberwithin{equation}{section}

\newcommand{\cR}{\mathcal R}
\newcommand{\E}{\mathbb E}
\newcommand{\ind}{\mathbf 1}
\allowdisplaybreaks[2]

\title[Large gaps between Romanoff numbers]{Large gaps between Romanoff numbers}
\author{Artyom Radomskii}
\address{HSE University, Moscow, Russian Federation}

\email{artyom.radomskii@mail.ru}
\date{}

\begin{document}
\begin{abstract}
Questions concerning representations of integers as the sum of a prime
and a power of two go back to the correspondence between Euler and
Goldbach in 1752 and to de Polignac's work of 1849. Romanoff proved
that the set of integers admitting such a representation has positive
lower density. Following the complementary direction
studied by Kalmynin and Konyagin, we consider long intervals containing
no Romanoff numbers. We use truncated products of divisor sums to
derive the bound $G_{\mathcal{R}}(X)\gg\log\log X$ for the longest block in
$[1,X]$ containing no integer of the form $p+2^n$, with $p$ prime and
$n\ge1$. This improves the bound
$G_{\mathcal{R}}(X)\gg\log\log X/\log\log\log X$ of Kalmynin and Konyagin.
The key step extends a translate lemma to sets of $O(\log X)$ integer
shifts of absolute value at most $X$: auxiliary primes dividing
differences of shifts are removed at negligible cost. After preliminary
sieving, this allows all remaining shifts to be covered simultaneously.
The argument is unconditional and also applies to $p+a^n$ for every
fixed integer $a\ge2$.
\end{abstract}
\maketitle

\section{Introduction}

Let $\PP$ denote the set of primes and $\cN=\{1,2,\ldots\}$. Put
\[
 \cR=\{p+2^n:p\in\PP,\ n\in\cN\}.
\]
The length of a block of consecutive integers means the number of
integers in the block. For $X\ge1$, define
\[
 \begin{aligned}
 G_{\cR}(X)=\max\bigl(&\{0\}\cup
 \{H\in\cN:\ \exists N\in\cZ_{\ge0},\ N+H\le X,\\
 &\hspace{33mm}\{N+1,\ldots,N+H\}\cap\cR=\varnothing\}\bigr).
 \end{aligned}
\]

The study of representations involving
primes and powers of two goes back to Euler and Goldbach; see Euler's
letter to Goldbach dated 16 December 1752~\cite{EulerGoldbach}.
De Polignac's paper of 1849~\cite[Th\'eor\`eme~2, IV]{Polignac} is
another early contribution to this subject. In 1934, Romanoff~\cite{Romanoff}
proved that $\cR$ has positive lower density. The theorem of Romanoff was generalized by many authors in several directions. Rieger~\cite{Rieger} extended Romanoff's theorem to algebraic number
fields. Shparlinski and Weingartner~\cite{ShparlinskiWeingartner}
proved an explicit analogue over finite fields, estimating the
proportion of polynomials representable as $h+g^k$, where $h$ is
irreducible and $g$ is a fixed polynomial of positive degree. Radomskii~\cite[Theorem~1.1]{Radomskii} established a general criterion
for many integers to have many representations $a_i+b_j$.
It combines counting and pair-correlation estimates for
$\cA=\{a_i\}$ with an averaged congruence condition for
$\cB=\{b_j\}$, explicitly allowing repeated terms in $\cB$.
Applications include prime summands and sums of two squares,
combined with polynomial values, restricted powers, and
elliptic-curve point counts.

The complementary problem concerns integers outside $\cR$.
Erd\H{o}s~\cite{Erdos} constructed, using covering congruences,
an arithmetic progression of odd integers containing no Romanoff
numbers. In particular, the upper asymptotic density of $\cR$ is
strictly less than $1/2$. Habsieger and Roblot~\cite[\S3]{HabsiegerRoblot}
subsequently obtained the explicit bound
\[
 \limsup_{X\to\infty}\frac{\#(\cR\cap[1,X])}{X}<0.490941
\]
by computations with residue classes.

Kalmynin and Konyagin~\cite[Theorem~1]{KK} proved that
\[
 G_{\cR}(X)\gg\frac{\log\log X}{\log\log\log X}.
\]

Using the weight construction in \cite[\S 4]{LG} we obtain the following improvement.

\begin{theorem}\label{thm:main}
There exist absolute constants $c>0$ and $X_0>0$ such that, for every
real $X\ge X_0$, the interval $[1,X]$ contains a block of at least
$\lfloor c\log\log X\rfloor$ consecutive integers not belonging to
$\cR$. In particular,
\begin{equation}\label{eq:main}
 G_{\cR}(X)\gg\log\log X.
\end{equation}
\end{theorem}

\begin{samepage}
\begin{corollary}\label{cor:base}
Let $a\ge2$ be a fixed integer and let
$\cR_a=\{p+a^n:p\in\PP,\ n\in\cN\}$. Then
\[
 G_{\cR_a}(X)\gg_a\log\log X
\]
for all sufficiently large $X$, where $G_{\cR_a}$ is defined
analogously to $G_{\cR}$.
\end{corollary}
\end{samepage}

We follow the construction from \cite{LG}, modifying some arguments in it. The product cutoff in this construction is analogous to the simplex restriction used in the Maynard--Tao
sieve; see Maynard~\cite[\S~7]{Maynard}. Proposition~1.2 of~\cite{LG} is stated for shifts of small diameter.
Our problem involves large negative shifts $j-2^n$, so a separate
adaptation of its proof is required. This is given in detail in
Section~\ref{sec:lemma}. 

\section{Notation}\label{sec:notation}

We write $\cN$ for the positive integers, $\cZ$ for the integers,
and $\PP$ for the primes. In particular, $\cN=\{1,2,\ldots\}$ and
$\cZ_{\ge0}=\{0,1,2,\ldots\}$.
The letter $p$ always denotes a prime in sums and products indexed by
$p$. All logarithms are natural, $\exp v=e^v$, and $\lfloor v\rfloor$
is the greatest integer not exceeding the real number $v$.
For integers $r,s$, the notation $(r,s)$ denotes their greatest common
divisor; $r\mid s$ and $r\nmid s$ mean, respectively, that $r$ divides
$s$ and that it does not. Congruences and residue classes modulo an
integer $q\ge1$ are written $r\equiv s\pmod q$ and $r\bmod q$.
For a finite set $S$, both $|S|$ and $\#S$ denote its cardinality.
The symbol $\varnothing$ denotes the empty set.
The binomial coefficient $\binom{k}{2}=k(k-1)/2$ counts unordered
pairs of distinct indices from $\{1,\ldots,k\}$.
Empty sums and products have values $0$ and $1$, respectively.

The indicator $\ind_{\mathcal E}$ is $1$ when the condition
$\mathcal E$ holds and $0$ otherwise. We use the arithmetic functions
\[
 \begin{gathered}
 \varphi(n)=\#\{1\le r\le n:(r,n)=1\},\qquad
 \omega(n)=\#\{p\in\PP:p\mid n\},\\
 \sigma_1(n)=\sum_{d\mid n}d\quad(n\in\cN).
 \end{gathered}
\]
Unless stated otherwise, divisors in sums and products are positive. We recall that $\varphi$ is called Euler's totient function.
Euler's constant is denoted by
\[
 \gamma_E=\lim_{n\to\infty}\left(\sum_{r=1}^n\frac1r-\log n\right).
\]

For $g>0$, the relations $f\ll g$ and $f=O(g)$ mean that
$|f|\le Cg$ for a suitable positive constant $C$ in the range under
consideration. For nonnegative quantities, $f\gg g$ means $g\ll f$,
and $f\asymp g$ means both $f\ll g$ and $g\ll f$.
A subscript, as in $\ll_a$ or $\gg_a$, permits dependence on the
fixed parameter $a$; otherwise the implied constants are absolute.
The relation $f=o(g)$ means $f/g\to0$, and $o(1)$ denotes a quantity
tending to zero. Limits are taken as $x\to\infty$ or $X\to\infty$,
as appropriate; the base $a$ is fixed. The constants $C_1\ge1$ and $\eta>0$ are absolute and,
once chosen, remain fixed throughout the proof.

\section{A lemma on large shifts}\label{sec:lemma}

\begin{lemma}\label{lem:translate}
There exists an absolute constant $\eta>0$ with the following property.
Let $x$ be sufficiently large, and suppose that
\[
 Q=\prod_{p\le x}p,
 \qquad s_1,\ldots,s_k\in\cZ,
 \qquad |s_i|\le e^{3x},
 \qquad k\le\eta x,
\]
where the shifts $s_i$ are pairwise distinct. Then, for every
$b\in\cZ$, there exists an integer $t$ with $1\le t\le e^x$ such
that each number $b+Qt+s_i$ is divisible by a prime in the interval
\[
 \left(x^6,\ \exp\!\left(\frac{x}{(\log x)^5}\right)\right].
\]
\end{lemma}

\begin{remark}
The lemma asserts divisibility only. Positivity of the values and the
fact that the resulting prime divisors are proper are checked when the
lemma is applied. 
\end{remark}

\begin{proof}
If $k=0$, there are no divisibility conditions, and $t=1$ proves
the assertion. Henceforth assume $k\ge1$. We may assume
that $\eta\le1$. Fix the parameters
\begin{equation}\label{eq:parameters}
 \begin{gathered}
 z=x^6,\qquad y=\exp\!\left(\frac{x}{(\log x)^5}\right),
 \qquad D=e^{x/4},\qquad T=\lfloor e^x\rfloor,\\
 \beta=\frac1{\log y},\qquad
 u=\frac{\log D}{\log y}=\frac14(\log x)^5.
 \end{gathered}
\end{equation}

\medskip\noindent\textit{Removing primes at which roots coincide.}
Put
\[
 \Delta=\prod_{1\le i<j\le k}|s_i-s_j|,
\]
where an empty product is understood to be $1$. Then
\[
 \log\Delta\le\binom{k}{2}(3x+\log2)=O(x^3).
\]
Define
\[
 P_0=\prod_{z<p\le y}p,\qquad
 P_* =\prod_{\substack{z<p\le y\\p\nmid\Delta}}p.
\]We have
\[
\#\{p: p>z, p|\Delta\}\le \sum_{\substack{p>z\\ p|\Delta}}\frac{\log p}{\log p}
\le \frac{1}{\log z}\sum_{p|\Delta}\log p\le \frac{\log \Delta}{\log z}.
\]Therefore
\begin{equation}\label{eq:exceptional}
 \sum_{\substack{p>z\\p\mid\Delta}}\frac1p
 \le\frac{\log\Delta}{z\log z}
 \ll\frac1{x^3\log x}.
\end{equation}
For $L_i(t)=b+Qt+s_i$ and every $p\mid P_*$, we have $(p,Q)=1$.
Thus $L_i(t)\equiv0\pmod p$ specifies exactly one residue class of $t$,
and these classes are distinct for different $i$.

\medskip\noindent\textit{Preserving the coefficient estimates.}
Introduce
\[
 B_0=\sum_{\substack{d\mid P_0\\d>1}}
       \frac1{\varphi(d)\log d},\qquad
 B=\sum_{\substack{d\mid P_*\\d>1}}
       \frac1{\varphi(d)\log d},\qquad
 U_0=\sum_{d\mid P_0}\frac1{\varphi(d)}.
\]
Let us show that
\begin{equation}\label{eq:full-estimates}
 B_0\asymp1,\qquad
 \sum_{\substack{d\mid P_0\\d>1}}
       \frac1{\varphi(d)(\log d)^2}\ll\frac1{(\log x)^2},
 \qquad U_0\asymp\frac{\log y}{\log z}.
\end{equation}

We shall use the following standard forms of Mertens' estimates:
\[
 \prod_{q\le t}\left(1-\frac1q\right)^{-1}
 \asymp\log t,\qquad
 \sum_{q\le t}\frac{\log q}{q}\ll\log t,
\]
and
\[
 \sum_{q\le t}\frac1q
 =\log\log t+M+O\!\left(\frac1{\log t}\right),
\]
where $M$ is an absolute constant and the letters $p$ and $q$
denote primes in these Mertens estimates.

\medskip\noindent
\textit{The estimate for $U_0$.}
Since $P_0$ is squarefree, multiplicativity gives
\[
 U_0
 =\sum_{d\mid P_0}\frac1{\varphi(d)}
 =\prod_{z<p\le y}\left(1+\frac1{p-1}\right)
 =\prod_{z<p\le y}\left(1-\frac1p\right)^{-1}.
\]
Consequently, Mertens' product estimate yields
\[
 U_0
 =\frac{\displaystyle
          \prod_{p\le y}(1-1/p)^{-1}}
        {\displaystyle
          \prod_{p\le z}(1-1/p)^{-1}}
 \asymp\frac{\log y}{\log z}.
\]

\medskip\noindent
\textit{Grouping divisors by their largest prime factor.}
For each prime $p\in(z,y]$, put
\[
 P_{<p}=\prod_{z<q<p}q,\qquad
 V_p=\sum_{e\mid P_{<p}}\frac1{\varphi(e)}.
\]
Again by multiplicativity and Mertens' product estimate,
\[
 V_p=\prod_{z<q<p}\left(1-\frac1q\right)^{-1}
 \asymp\frac{\log p}{\log z},
\]
uniformly for $z<p\le y$. Indeed,
\[
 V_p
 =\left(1-\frac1p\right)
   \frac{\displaystyle
          \prod_{q\le p}(1-1/q)^{-1}}
        {\displaystyle
          \prod_{q\le z}(1-1/q)^{-1}}.
\]

Every divisor $d>1$ of $P_0$ has a unique representation
$d=pe$, where $p$ is its largest prime factor and
$e\mid P_{<p}$. Since $(p,e)=1$, we have
$\varphi(pe)=(p-1)\varphi(e)$. Thus
\[
 B_0
 =\sum_{z<p\le y}\frac1{p-1}
   \sum_{e\mid P_{<p}}
       \frac1{\varphi(e)\log(pe)}.
\]

\medskip\noindent
\textit{Upper and lower bounds for the inner sum.}
Write
\[
 T_p=\sum_{e\mid P_{<p}}
          \frac1{\varphi(e)\log(pe)}.
\]
The inequality $\log(pe)\ge\log p$ immediately gives
\[
 T_p\le\frac{V_p}{\log p}.
\]

For the lower bound, we first evaluate a logarithmic moment.
Fix a prime $q\mid P_{<p}$. Since $P_{<p}$ is squarefree,
every divisor $e\mid P_{<p}$ divisible by $q$ can be written
uniquely as
\[
 e=qf,\qquad f\mid P_{<p}/q,\qquad (q,f)=1.
\]
By the multiplicativity of $\varphi$, we have
$\varphi(qf)=(q-1)\varphi(f)$. Therefore
\[
 \sum_{\substack{e\mid P_{<p}\\q\mid e}}
       \frac1{\varphi(e)}
 =\frac1{q-1}
   \sum_{f\mid P_{<p}/q}\frac1{\varphi(f)}.
\]
We see that
\[
 \sum_{f\mid P_{<p}/q}\frac1{\varphi(f)}
 =\prod_{\substack{z<r<p\\r\in\PP\\r\ne q}}
       \left(1+\frac1{r-1}\right).
\]
On the other hand, separating the Euler factor corresponding
to $q$ in the product for $V_p$, we obtain
\[
 \begin{aligned}
 V_p
 &=\prod_{\substack{z<r<p\\r\in\PP}}\left(1+\frac1{r-1}\right)\\
 &=\left(1+\frac1{q-1}\right)
   \prod_{\substack{z<r<p\\r\in\PP\\r\ne q}}
       \left(1+\frac1{r-1}\right)\\
 &=\frac{q}{q-1}
   \prod_{\substack{z<r<p\\r\in\PP\\r\ne q}}
       \left(1+\frac1{r-1}\right).
 \end{aligned}
\]
Consequently,
\[
 \begin{aligned}
 \sum_{\substack{e\mid P_{<p}\\q\mid e}}
       \frac1{\varphi(e)}
 &=\frac1{q-1}
   \prod_{\substack{z<r<p\\r\in\PP\\r\ne q}}
       \left(1+\frac1{r-1}\right)\\
 &=\frac1{q-1}\cdot\frac{q-1}{q}\,V_p
 =\frac{V_p}{q}.
 \end{aligned}
\]
Since all divisors of $P_{<p}$ are squarefree, it follows that
\[
 \begin{aligned}
 \sum_{e\mid P_{<p}}\frac{\log e}{\varphi(e)}
 &=
 \sum_{z<q<p}\log q
   \sum_{\substack{e\mid P_{<p}\\q\mid e}}
       \frac1{\varphi(e)}\\
 &=V_p\sum_{z<q<p}\frac{\log q}{q}
 \ll V_p\log p.
 \end{aligned}
\]
Therefore
\[
 \sum_{e\mid P_{<p}}\frac{\log(pe)}{\varphi(e)}
 =V_p\left(\log p+
          \sum_{z<q<p}\frac{\log q}{q}\right)
 \ll V_p\log p.
\]

Applying the Cauchy--Schwarz inequality, we obtain
\[
 \begin{aligned}
 V_p^2
 &=
 \left(
   \sum_{e\mid P_{<p}}
   \sqrt{\frac1{\varphi(e)\log(pe)}}
   \sqrt{\frac{\log(pe)}{\varphi(e)}}
 \right)^2\\
 &\le
 T_p\sum_{e\mid P_{<p}}
          \frac{\log(pe)}{\varphi(e)}
 \ll T_pV_p\log p.
 \end{aligned}
\]
Hence
\[
 T_p\gg\frac{V_p}{\log p}.
\]
Combining the two bounds, we conclude that
\[
 T_p\asymp\frac{V_p}{\log p}
 \asymp\frac1{\log z},
\]
uniformly for $z<p\le y$.

\medskip\noindent
\textit{The estimate for $B_0$.}
The preceding calculation gives
\[
 B_0
 \asymp\frac1{\log z}
       \sum_{z<p\le y}\frac1{p-1}.
\]
Since
\[
 \sum_{z<p\le y}
 \left(\frac1{p-1}-\frac1p\right)
 \ll\sum_{n>z}\frac1{n^2}
 \ll\frac1z,
\]
Mertens' estimate implies
\[
 \sum_{z<p\le y}\frac1{p-1}
 =\log\!\left(\frac{\log y}{\log z}\right)
  +O\!\left(\frac1{\log z}\right).
\]
For our choice of parameters,
\[
 \frac{\log y}{\log z}
 =\frac{x}{6(\log x)^6},
\]
and consequently
\[
 \frac1{\log z}
 \log\!\left(\frac{\log y}{\log z}\right)
 =
 \frac{\log x-6\log\log x-\log6}{6\log x}
 \longrightarrow\frac16.
\]
It follows that $B_0\asymp1$.

\medskip\noindent
\textit{The sum with $(\log d)^2$ in the denominator.}
Grouping divisors by their largest prime factor once more,
we obtain
\[
 \begin{aligned}
 \sum_{\substack{d\mid P_0\\d>1}}
       \frac1{\varphi(d)(\log d)^2}
 &=
 \sum_{z<p\le y}\frac1{p-1}
   \sum_{e\mid P_{<p}}
       \frac1{\varphi(e)(\log(pe))^2}\\
 &\le
 \sum_{z<p\le y}
       \frac{V_p}{(p-1)(\log p)^2}\\
 &\ll
 \frac1{\log z}
 \sum_{z<p\le y}\frac1{p\log p}.
 \end{aligned}
\]

To estimate the last prime sum, define
\[
 H_z(t)=\sum_{z<p\le t}\frac{\log p}{p}
 \qquad(t\ge z).
\]
Mertens' estimate gives $H_z(t)\ll\log t$.
Partial summation therefore yields
\[
 \begin{aligned}
 \sum_{p>z}\frac1{p\log p}
 &=
 \lim_{R\to\infty}
 \left(
   \frac{H_z(R)}{(\log R)^2}
   +2\int_z^R
       \frac{H_z(t)}{t(\log t)^3}\,dt
 \right)\\
 &\ll
 \int_z^\infty\frac{dt}{t(\log t)^2}
 =\frac1{\log z}.
 \end{aligned}
\]
Thus
\[
 \sum_{\substack{d\mid P_0\\d>1}}
       \frac1{\varphi(d)(\log d)^2}
 \ll\frac1{(\log z)^2}
 \ll\frac1{(\log x)^2},
\]
because $\log z=6\log x$. This proves all three estimates
in~\eqref{eq:full-estimates}.

For every prime $p\mid P_0$, we have
\begin{equation}\label{eq:U0/p}
 \sum_{\substack{d\mid P_0\\p\mid d}}\frac1{\varphi(d)}
 =\frac{U_0}{p}.
\end{equation}
Indeed, since $P_0$ is squarefree, every divisor $d\mid P_0$
divisible by $p$ can be written uniquely as $d=pe$, where
$e\mid P_0/p$ and $(p,e)=1$. Hence
$\varphi(pe)=(p-1)\varphi(e)$, and therefore
\[
 \sum_{\substack{d\mid P_0\\p\mid d}}\frac1{\varphi(d)}
 =\frac1{p-1}\sum_{e\mid P_0/p}\frac1{\varphi(e)}.
\]
On the other hand, separating the divisors of $P_0$ according
to whether they are divisible by $p$, we obtain
\[
 \begin{aligned}
 U_0
 &=\sum_{e\mid P_0/p}\frac1{\varphi(e)}
   +\sum_{e\mid P_0/p}\frac1{\varphi(pe)}\\
 &=\left(1+\frac1{p-1}\right)
   \sum_{e\mid P_0/p}\frac1{\varphi(e)}
 =\frac{p}{p-1}\sum_{e\mid P_0/p}\frac1{\varphi(e)}.
 \end{aligned}
\]
Consequently,
\[
 \sum_{\substack{d\mid P_0\\p\mid d}}\frac1{\varphi(d)}
 =\frac1{p-1}\cdot\frac{p-1}{p}\,U_0
 =\frac{U_0}{p},
\]and \eqref{eq:U0/p} is proved.

Let
\[
 \mathcal E=\{p\in\PP:z<p\le y,\ p\mid\Delta\}
\]
be the set of removed primes. Since $P_*\mid P_0$, we have
\[
 B_0-B
 =\sum_{\substack{d\mid P_0\\d\nmid P_*}}
       \frac1{\varphi(d)\log d}.
\]
Every divisor occurring in this sum is divisible by at least
one prime in $\mathcal E$. All summands are nonnegative, so
\[
 0\le B_0-B
 \le\sum_{p\in\mathcal E}
      \sum_{\substack{d\mid P_0\\p\mid d}}
          \frac1{\varphi(d)\log d}.
\]
A divisor containing several primes from $\mathcal E$ is counted
more than once on the right, which only increases the sum.

For each term in the inner sum, we have $d\ge p>z$, and hence
$1/\log d\le1/\log z$. Using \eqref{eq:U0/p}, \eqref{eq:exceptional}, \eqref{eq:full-estimates}, we obtain
\begin{equation}\label{eq:B-loss}
 \begin{aligned}
 0\le B_0-B
 &\le\frac1{\log z}
      \sum_{p\in\mathcal E}
      \sum_{\substack{d\mid P_0\\p\mid d}}
          \frac1{\varphi(d)}\\
 &=\frac{U_0}{\log z}
      \sum_{p\in\mathcal E}\frac1p\\
 &\ll\frac{\log y}{(\log z)^2}
      \frac1{x^3\log x}\\
 &\ll\frac1{x^2(\log x)^8}=o(1).
 \end{aligned}
\end{equation}
Since $B_0\asymp1$ and $B_0-B=o(1)$ by~\eqref{eq:B-loss}, we conclude that
$B\asymp1$. In particular, $P_*>1$ for sufficiently large $x$.

Put
\[
 \begin{gathered}
 a(1)=1,\qquad a(d)=-\frac1{B\log d}\quad(d\mid P_*,\ d>1),\\
 A_\gamma=\sum_{d\mid P_*}\frac{a(d)^2d^\gamma}{\varphi(d)},
 \qquad A=A_0.
 \end{gathered}
\]
\medskip\noindent\textit{The basic coefficient identities.}
We first prove each of the three assertions
\begin{equation}\label{eq:basic-coefficients}
 \sum_{d\mid P_*}\frac{a(d)}{\varphi(d)}=0,
 \qquad |a(d)|\le1,
 \qquad A=1+O((\log x)^{-2}).
\end{equation}
Here and below, a statement about $a(d)$ is understood to concern
divisors $d$ of $P_*$. Since $\varphi(1)=1$, the definition of $B$
and that of $a(d)$ give the exact cancellation
\[
 \sum_{d\mid P_*}\frac{a(d)}{\varphi(d)}
 =1-\frac1B\sum_{\substack{d\mid P_*\\d>1}}
                  \frac1{\varphi(d)\log d}
 =1-\frac BB=0.
\]
Next, $|a(1)|=1$. If $d>1$ divides $P_*$, then $d$ has a prime
factor greater than $z$, and consequently $d>z$. We have already
proved that $B\asymp1$, so there is an absolute constant $b_1>0$
such that $B\ge b_1$ for all sufficiently large $x$. Therefore
\[
 |a(d)|=\frac1{B\log d}
 \le\frac1{b_1\log z}
 =\frac1{6b_1\log x}\le1
 \qquad(d>1)
\]
once $x$ is sufficiently large.

Finally, separating the term $d=1$ in $A$ and using $P_*\mid P_0$,
we obtain
\[
 \begin{aligned}
 0\le A-1
 &=\frac1{B^2}\sum_{\substack{d\mid P_*\\d>1}}
                  \frac1{\varphi(d)(\log d)^2}\\
 &\le\frac1{B^2}\sum_{\substack{d\mid P_0\\d>1}}
                  \frac1{\varphi(d)(\log d)^2}
 \ll\frac1{(\log x)^2}.
 \end{aligned}
\]
The last estimate uses~\eqref{eq:full-estimates} and $B\asymp1$.
This proves~\eqref{eq:basic-coefficients}; in particular, $A\ge1$.

\medskip\noindent\textit{Estimates with a small power of the divisor.}
For $0\le\gamma\le2\beta$, define
\[
 U_\gamma=\sum_{d\mid P_0}\frac{d^\gamma}{\varphi(d)}.
\]
The function $d^\gamma/\varphi(d)$ is multiplicative, and $P_0$ is
squarefree. Expanding its Euler product and then dividing by the
corresponding product for $U_0$, we find
\[
 \begin{aligned}
 U_\gamma&=\prod_{z<p\le y}
                    \left(1+\frac{p^\gamma}{p-1}\right),\\
 \frac{U_\gamma}{U_0}
 &=\prod_{z<p\le y}
     \frac{1+p^\gamma/(p-1)}{1+1/(p-1)}
 =\prod_{z<p\le y}\left(1+\frac{p^\gamma-1}{p}\right).
 \end{aligned}
\]
For every prime in this product, $0\le\gamma\log p\le2$.
The identity $e^v-1=\int_0^v e^s\,ds$ therefore gives
\[
 0\le p^\gamma-1
 =e^{\gamma\log p}-1
 \le e^2\gamma\log p.
\]
Applying $1+h\le e^h$ for $h\ge0$ to the Euler factors, and then
using Mertens' estimate, yields
\[
 \frac{U_\gamma}{U_0}
 \le\exp\!\left(\sum_{z<p\le y}\frac{p^\gamma-1}{p}\right)
 \le\exp\!\left(e^2\gamma\sum_{p\le y}\frac{\log p}{p}\right)
 \le\exp(C\gamma\log y)\ll1.
\]
The last bound is uniform in $\gamma$, since $\gamma\log y\le2$.
Consequently, $U_\gamma\ll U_0\ll\log y/\log z$.

We next estimate the sum involving $|a(d)|$ by separating the
ranges $1<d\le y$ and $d>y$. In the first range, $d^\gamma\le e^2$,
and the definition of $B$ gives
\[
 \begin{aligned}
 \sum_{\substack{d\mid P_*\\1<d\le y}}
       \frac{|a(d)|d^\gamma}{\varphi(d)}
 &\le\frac{e^2}{B}
       \sum_{\substack{d\mid P_*\\1<d\le y}}
            \frac1{\varphi(d)\log d}\\
 &\le\frac{e^2}{B}
       \sum_{\substack{d\mid P_*\\d>1}}
            \frac1{\varphi(d)\log d}=e^2.
 \end{aligned}
\]
In the second range, $1/\log d\le1/\log y$. Since $P_*\mid P_0$,
we have
\[
 \begin{aligned}
 \sum_{\substack{d\mid P_*\\d>y}}
       \frac{|a(d)|d^\gamma}{\varphi(d)}
 &\le\frac1{B\log y}
       \sum_{d\mid P_*}\frac{d^\gamma}{\varphi(d)}\\
 &\le\frac{U_\gamma}{B\log y}
 \ll\frac1{\log z}\ll 1.
 \end{aligned}
\]
Adding the two estimates proves
\[
 \sum_{\substack{d\mid P_*\\d>1}}
       \frac{|a(d)|d^\gamma}{\varphi(d)}
 \ll 1.
\]

To control $A_\gamma$, observe that for $d\ge1$,
\[
 d^\gamma-1
 =\int_0^\gamma d^s\log d\,ds
 \le\gamma d^\gamma\log d.
\]
Also, for $d>1$, the definition of $a(d)$ implies
$a(d)^2\log d=|a(d)|/B$. It follows that
\[
 \begin{aligned}
 0\le A_\gamma-A
 &=\sum_{\substack{d\mid P_*\\d>1}}
       \frac{a(d)^2(d^\gamma-1)}{\varphi(d)}\\
 &\le\gamma\sum_{\substack{d\mid P_*\\d>1}}
       \frac{a(d)^2d^\gamma\log d}{\varphi(d)}\\
 &=\frac\gamma B\sum_{\substack{d\mid P_*\\d>1}}
       \frac{|a(d)|d^\gamma}{\varphi(d)}
 \ll\gamma.
 \end{aligned}
\]
Thus there is an absolute constant $C>0$ such that
\begin{equation}\label{eq:tilt}
 \sum_{\substack{d\mid P_*\\d>1}}
       \frac{|a(d)|d^\gamma}{\varphi(d)}\le C,
 \qquad A_\gamma\le A+C\gamma
 \qquad(0\le\gamma\le2\beta).
\end{equation}
For later use, fix an absolute constant $C_1\ge1$ such that
$A_\gamma\le A+C_1\gamma$ throughout this range.

\medskip\noindent\textit{Divisors containing a prescribed prime.}
Fix a prime $p\mid P_*$. Every divisor $d\mid P_*$ with $p\mid d$
has a unique representation $d=pe$, where $e\mid P_*/p$ and
$(p,e)=1$. Separating $e=1$ and using
$\varphi(pe)=(p-1)\varphi(e)$, we obtain
\[
 \begin{aligned}
 \sum_{\substack{d\mid P_*\\p\mid d}}
       \frac{|a(d)|}{\varphi(d)}
 &=\frac1{B(p-1)}
   \sum_{e\mid P_*/p}\frac1{\varphi(e)\log(pe)}\\
 &=\frac1{B(p-1)}\left(
      \frac1{\log p}
      +\sum_{\substack{e\mid P_*/p\\e>1}}
          \frac1{\varphi(e)\log(pe)}\right).
 \end{aligned}
\]
For $e>1$, we have $\log(pe)\ge\log e$, while every divisor of
$P_*/p$ is also a divisor of $P_*$. Hence
\[
 \sum_{\substack{e\mid P_*/p\\e>1}}
       \frac1{\varphi(e)\log(pe)}
 \le\sum_{\substack{e\mid P_*\\e>1}}
       \frac1{\varphi(e)\log e}=B.
\]
Substitution proves
\begin{equation}\label{eq:prime-coefficients}
 \sum_{\substack{d\mid P_*\\p\mid d}}
       \frac{|a(d)|}{\varphi(d)}
 \le\frac1{B(p-1)}\left(\frac1{\log p}+B\right)
 \ll\frac1p.
\end{equation}
Indeed, $B$ is bounded below by an absolute positive constant,
$\log p\ge\log z$, and $1/(p-1)\le2/p$.

\medskip\noindent\textit{Constructing the local factors and computing their moments.}
For $p\mid P_*$ and $d\mid P_*$, define
\[
 \psi_{i,p}(t)=\frac{1-p\ind_{p\mid L_i(t)}}{p-1},
 \qquad
 \psi_{i,d}(t)=\prod_{p\mid d}\psi_{i,p}(t),
 \qquad \psi_{i,1}(t)=1.
\]
The last convention agrees with the empty-product convention.
For a $P_*$-periodic function $F$, write
\[
 \E F=\frac1{P_*}\sum_{t=0}^{P_*-1}F(t).
\]

Fix a prime $p\mid P_*$. In the complete average $\E$, every
integer $t\in\{0,\ldots,P_*-1\}$ is assigned the same weight
$1/P_*$. For each residue $r\in\{0,\ldots,p-1\}$, the integers
in this range satisfying $t\equiv r\pmod p$ are precisely
\[
 r,\ r+p,\ \ldots,\
 r+\left(\frac{P_*}{p}-1\right)p.
\]
There are exactly $P_*/p$ such integers. Consequently,
\[
 \E\,\ind_{t\equiv r\pmod p}
 =\frac1{P_*}\cdot\frac{P_*}{p}
 =\frac1p.
\]
Thus all $p$ residue classes modulo $p$ have the same weight
under $\E$. This is the meaning of uniform distribution
modulo $p$ in the present setting.

More generally, for any function $f$ that is periodic modulo
$p$, grouping the terms according to their residues modulo
$p$ gives
\[
 \begin{aligned}
 \E f
 &=\frac1{P_*}\sum_{t=0}^{P_*-1}f(t)\\
 &=\frac1{P_*}\sum_{r=0}^{p-1}
       \sum_{h=0}^{P_*/p-1}f(r+hp)\\
 &=\frac1{P_*}\sum_{r=0}^{p-1}
       \frac{P_*}{p}f(r)
 =\frac1p\sum_{r=0}^{p-1}f(r).
 \end{aligned}
\]
In particular, this identity applies to the local factors
$\psi_{i,p}$ and their products.

We next explain why the residue coordinates for distinct primes
are independent under this average. Let $p_1,\ldots,p_\ell$
be distinct prime divisors of $P_*$, and fix arbitrary residues
$a_\nu\bmod p_\nu$. By the Chinese remainder theorem, the system
\[
 t\equiv a_\nu\pmod{p_\nu},\qquad 1\le\nu\le\ell,
\]
specifies exactly one residue class modulo $q=p_1\cdots p_\ell$.
Since $q\mid P_*$, this system has exactly $P_*/q$ solutions
among the integers $0\le t<P_*$. Consequently,
\[
 \E\prod_{\nu=1}^{\ell}
       \ind_{t\equiv a_\nu\pmod{p_\nu}}
 =\frac1{p_1\cdots p_\ell}
 =\prod_{\nu=1}^{\ell}
       \E\,\ind_{t\equiv a_\nu\pmod{p_\nu}}.
\]
Thus, when $t$ is chosen uniformly modulo $P_*$, the residues
$t\bmod p$, with $p\mid P_*$, are mutually independent and
each is uniformly distributed modulo $p$.

More generally, let $f_p(t)$ be any function depending only
on $t\bmod p$. Since $P_*$ is squarefree, the Chinese remainder
theorem gives a bijection
\[
 \mathbb Z/P_*\mathbb Z
 \longrightarrow
 \prod_{p\mid P_*}\mathbb Z/p\mathbb Z,
 \qquad
 t\longmapsto(t\bmod p)_{p\mid P_*}.
\]
Hence the average of a product of such functions factors as
\[
 \begin{aligned}
 \E\prod_{p\mid P_*}f_p(t)
 &=\frac1{P_*}
   \sum_{\substack{0\le a_p<p\\\text{for each }p\mid P_*}}
      \prod_{p\mid P_*}f_p(a_p)\\
 &=\prod_{p\mid P_*}
      \left(\frac1p\sum_{a=0}^{p-1}f_p(a)\right)
 =\prod_{p\mid P_*}\E f_p(t).
 \end{aligned}
\]
Each function $\psi_{i,p}(t)$ depends only on $t\bmod p$.
Therefore, to average a product of these functions, we first
group together all factors associated with the same prime $p$,
and then factor the average over the distinct primes.
We now compute the local moments needed for these factorizations.

Since $(Q,p)=1$, multiplication by $Q$ permutes the residue
classes modulo $p$. Therefore the congruence
\[
 L_i(t)=b+Qt+s_i\equiv0\pmod p
\]
has exactly one solution modulo $p$. By the definition of
$\psi_{i,p}$,
\[
 \psi_{i,p}(t)=
 \begin{cases}
  -1,&p\mid L_i(t),\\[1mm]
  \dfrac1{p-1},&p\nmid L_i(t).
 \end{cases}
\]
Hence $\psi_{i,p}$ takes the value $-1$ on one residue class
modulo $p$, of weight $1/p$, and the value $1/(p-1)$ on the
other $p-1$ classes, of total weight $(p-1)/p$.
Direct calculation gives
\[
 \E\psi_{i,p}
 =\frac1p\left(-1+(p-1)\frac1{p-1}\right)=0
\]
and
\[
 \E\psi_{i,p}^2
 =\frac1p\left(1+(p-1)\frac1{(p-1)^2}\right)
 =\frac1{p-1}.
\]
If $i\ne j$, the roots of $L_i$ and $L_j$ modulo $p$ are distinct:
a common root would imply $p\mid s_i-s_j$, contrary to
$p\nmid\Delta$. On each of these two root classes,
$\psi_{i,p}\psi_{j,p}=-1/(p-1)$; on the remaining $p-2$ classes,
the product is $1/(p-1)^2$. Therefore
\[
 \E(\psi_{i,p}\psi_{j,p})
 =\frac1p\left(-\frac2{p-1}+\frac{p-2}{(p-1)^2}\right)
 =-\frac1{(p-1)^2}.
\]
We have proved all the local moment identities
\begin{equation}\label{eq:local-moments}
 \E\psi_{i,p}=0,\qquad
 \E\psi_{i,p}^2=\frac1{p-1},\qquad
 \E(\psi_{i,p}\psi_{j,p})=-\frac1{(p-1)^2}\quad(i\ne j).
\end{equation}
The independence established above applies only to distinct prime
moduli. For a fixed prime $p$, the functions $\psi_{i,p}$ and
$\psi_{j,p}$ need not be independent. Indeed, for $i\ne j$,
\eqref{eq:local-moments} gives
\[
 \E(\psi_{i,p}\psi_{j,p})
 =-\frac1{(p-1)^2}
 \ne 0
 =(\E\psi_{i,p})(\E\psi_{j,p}).
\]

\medskip\noindent\textit{The truncated weight.}
For $0\le j\le k$, let
\[
 \mathcal D_j=\{(r_1,\ldots,r_j):r_i\mid P_*,\ (r_i,r_\ell)=1
 \ (i\ne\ell),\ r_1\cdots r_j\le D\},
\]
where $\mathcal D_0$ contains only the empty tuple, whose product
is $1$. The coprimality condition ensures that each prime occurs
in at most one coordinate, while the product condition bounds the
moduli that arise when the weight is expanded. Define
\[
 R(t)=\sum_{\bm r\in\mathcal D_k}
             \prod_{i=1}^k a(r_i)\psi_{i,r_i}(t),
 \qquad w(t)=R(t)^2,
\]
and
\[
 S_1=\frac1T\sum_{t=1}^T w(t),\qquad
 S_2=\frac1T\sum_{t=1}^T w(t)
                \sum_{i=1}^k\ind_{(L_i(t),P_*)=1}.
\]
We call $L_i(t)$ uncovered if $(L_i(t),P_*)=1$. The inner sum in
$S_2$ is the number of uncovered forms, a nonnegative integer.
If it were at least $1$ for every $t$, the nonnegativity of $w(t)$
would imply $S_2\ge S_1$. Thus $S_2<S_1$ will prove that some
$1\le t\le T$ has no uncovered form.

\medskip\noindent\textit{Expansion into divisibility indicators.}
For a squarefree divisor $r\mid P_*$, expanding the product defining
$\psi_{i,r}$ gives
\[
 \begin{aligned}
 \psi_{i,r}(t)
 &=\frac1{\varphi(r)}
           \prod_{p\mid r}(1-p\ind_{p\mid L_i(t)})\\
 &=\frac1{\varphi(r)}
           \sum_{d\mid r}(-1)^{\omega(d)}d\,
                       \ind_{d\mid L_i(t)}.
 \end{aligned}
\]
Here a subset of the prime factors of $r$ determines its product
$d$, and the corresponding product of indicators is exactly
$\ind_{d\mid L_i(t)}$. If $d_i\mid r_i$ for a tuple
$\bm r\in\mathcal D_k$, then the $d_i$ are pairwise coprime and
$d_1\cdots d_k\le r_1\cdots r_k\le D$. Hence, after collecting
coefficients, the expansion has the form
\[
 R(t)=\sum_{\bm d\in\mathcal D_k}\lambda_{\bm d}
                         \prod_{i=1}^k\ind_{d_i\mid L_i(t)}.
\]
We estimate the sum of absolute coefficients before collecting
equal indicator products. The triangle inequality gives
\[
 \begin{aligned}
 \sum_{\bm d\in\mathcal D_k}|\lambda_{\bm d}|
 &\le\sum_{\bm r\in\mathcal D_k}
       \prod_{i=1}^k\left(
          \frac{|a(r_i)|}{\varphi(r_i)}\sum_{d_i\mid r_i}d_i
       \right)\\
 &=\sum_{\bm r\in\mathcal D_k}
       \prod_{i=1}^k\frac{|a(r_i)|\sigma_1(r_i)}{\varphi(r_i)}\\
 &\le\sum_{\bm r\in\mathcal D_k}
       \frac{\sigma_1(r_1\cdots r_k)}{\varphi(r_1\cdots r_k)}.
 \end{aligned}
\]
The last step uses $|a(r_i)|\le1$ and the multiplicativity of
$\sigma_1$ and $\varphi$ on pairwise coprime arguments.

Fix a divisor $n\mid P_*$ with $n\le D$. We count the ordered
tuples $(r_1,\ldots,r_k)\in\mathcal D_k$ satisfying
$r_1\cdots r_k=n$. Suppose first that $n>1$. Since $P_*$
is squarefree, we may write
\[
 n=p_1\cdots p_\ell,\qquad \ell=\omega(n),
\]
where the primes $p_1,\ldots,p_\ell$ are distinct.

In every such tuple, each prime $p_j$ divides exactly one
coordinate. Indeed, it divides at least one coordinate because
$r_1\cdots r_k=n$, and it cannot divide two coordinates because
the coordinates are pairwise coprime. Thus the tuple determines
a map
\[
 f:\{1,\ldots,\ell\}\longrightarrow\{1,\ldots,k\},
\]
where $f(j)=i$ means that $p_j\mid r_i$.

Conversely, every such map determines a tuple by
\[
 r_i=\prod_{\substack{1\le j\le\ell\\f(j)=i}}p_j
 \qquad(1\le i\le k),
\]
with an empty product interpreted as $1$. These coordinates
are pairwise coprime divisors of $P_*$, and their product is
$n\le D$. Hence the resulting tuple belongs to $\mathcal D_k$.
The two constructions are inverse to each other, so this is
a bijection.

For each of the $\ell$ primes, there are $k$ independent
choices of its coordinate. Therefore
\[
 \#\{(r_1,\ldots,r_k)\in\mathcal D_k:r_1\cdots r_k=n\}
 =k^\ell=k^{\omega(n)}.
\]
For $n=1$, the only possible tuple is $(1,\ldots,1)$.
Since $\omega(1)=0$ and $k\ge1$, the same formula gives
$k^{\omega(1)}=1$.

It follows that
\[
 \sum_{\bm d\in\mathcal D_k}|\lambda_{\bm d}|
 \le\sum_{\substack{n\mid P_*\\n\le D}}
             k^{\omega(n)}\frac{\sigma_1(n)}{\varphi(n)}.
\]

We now bound both factors in each summand. All prime divisors of
$n$ exceed $z$, so
\begin{equation}\label{w(n).calc}
 \omega(n)\log z\le\sum_{p\mid n}\log p
 =\log n\le\log D.
\end{equation}
Since $n$ is squarefree and $p>z\ge3$, we also have
\[
 \begin{aligned}
 \frac{\sigma_1(n)}{\varphi(n)}
 &=\prod_{p\mid n}\frac{p+1}{p-1}
 =\prod_{p\mid n}\left(1+\frac2{p-1}\right)\\
 &\le\exp\!\left(\sum_{p\mid n}\frac2{p-1}\right)
 \le\exp\!\left(\frac{3\omega(n)}z\right)
 \le\exp\!\left(\frac{3\log D}{z\log z}\right)\le2
 \end{aligned}
\]
for sufficiently large $x$. The estimate includes $n=1$.
Furthermore, $k\ge1$, so
\[
 k^{\omega(n)}
 \le\exp\!\left(\frac{\log D\log k}{\log z}\right)
 =D^{\log k/\log z}.
\]
Since $k\le\eta x\le x$ and $\log z=6\log x$,
we have $\log k/\log z\le1/6$. Thus
\begin{equation}\label{eq:l1}
 \sum_{\bm d\in\mathcal D_k}|\lambda_{\bm d}|
 \le\sum_{\substack{n\mid P_*\\n\le D}}
             k^{\omega(n)}\frac{\sigma_1(n)}{\varphi(n)}
 \le2D^{1+\log k/\log z}\le D^{3/2}.
\end{equation}
The last inequality follows from $2D^{7/6}\le D^{3/2}$ for
$D\ge8$, which holds once $x$ is sufficiently large.

\medskip\noindent\textit{Comparison of the two averages.}
For $\bm d\in\mathcal D_k$, write
$J_{\bm d}(t)=\prod_{i=1}^k\ind_{d_i\mid L_i(t)}$.
Expanding the square of $R$ gives
\[
 R(t)^2=\sum_{\bm d,\bm e\in\mathcal D_k}
               \lambda_{\bm d}\lambda_{\bm e}
               J_{\bm d}(t)J_{\bm e}(t).
\]

Fix two tuples $\bm d,\bm e\in\mathcal D_k$. By definition,
\[
 J_{\bm d}(t)J_{\bm e}(t)
 =\prod_{i=1}^k
       \ind_{d_i\mid L_i(t)}\ind_{e_i\mid L_i(t)}.
\]
Thus this product equals $1$ if and only if all the conditions
\[
 d_i\mid L_i(t),\qquad e_i\mid L_i(t)
 \qquad(1\le i\le k)
\]
hold simultaneously. Put
\[
 q=\operatorname{lcm}(d_1,\ldots,d_k,e_1,\ldots,e_k),
\]
where $\operatorname{lcm}$ denotes the least common multiple.
Since all the coordinates divide $P_*$, we have $q\mid P_*$.

Suppose first that $(d_i,e_j)>1$ for some distinct indices
$i$ and $j$. Choose a prime $p\mid(d_i,e_j)$. If the product
$J_{\bm d}(t)J_{\bm e}(t)$ were equal to $1$, then
\[
 p\mid L_i(t),\qquad p\mid L_j(t),
\]
and consequently
\[
 p\mid L_i(t)-L_j(t)=s_i-s_j.
\]
This is impossible: $p\mid P_*$, whereas every prime divisor
of a difference $s_i-s_j$ has been excluded from $P_*$.
Therefore
\[
 J_{\bm d}(t)J_{\bm e}(t)=0
 \qquad\text{for every }t\in\cZ.
\]

Suppose now that
\[
 (d_i,e_j)=1\qquad\text{for all }i\ne j.
\]
For each $i$, define
\[
 m_i=\operatorname{lcm}(d_i,e_i).
\]
The two divisibility conditions for the $i$th form are
equivalent to a single condition:
\[
 d_i\mid L_i(t)\ \text{and}\ e_i\mid L_i(t)
 \quad\Longleftrightarrow\quad
 m_i\mid L_i(t).
\]
Moreover, since $\bm d,\bm e\in\mathcal D_k$, their coordinates
are pairwise coprime within each tuple. Together with the
assumed cross-coprimality, this gives
\[
 (d_i,d_j)=(e_i,e_j)=(d_i,e_j)=(e_i,d_j)=1
 \qquad(i\ne j).
\]
Hence $(m_i,m_j)=1$ whenever $i\ne j$, and therefore
\[
 m_1\cdots m_k
 =\operatorname{lcm}(m_1,\ldots,m_k)
 =q.
\]

Each $m_i$ divides $P_*$, so $(Q,m_i)=1$. It follows that
the congruence
\[
 b+Qt+s_i\equiv0\pmod{m_i}
\]
determines exactly one residue class modulo $m_i$.
When $m_i=1$, this condition is satisfied by every integer.
Since the moduli $m_i$ are pairwise coprime, the Chinese
remainder theorem shows that the entire system is equivalent
to one congruence
\[
 t\equiv a\pmod q
\]
for a uniquely determined residue class $a\bmod q$.
Consequently,
\[
 J_{\bm d}(t)J_{\bm e}(t)=
 \begin{cases}
  0,
  &\text{if }(d_i,e_j)>1\text{ for some }i\ne j,\\[1mm]
  \ind_{t\equiv a\pmod q},
  &\text{if }(d_i,e_j)=1\text{ for all }i\ne j.
 \end{cases}
\]
If $q=1$, the indicator in the second case is identically $1$.

We now compare the complete average with the average over
$1\le t\le T$. In the first case, both averages vanish.
In the second case, since $q\mid P_*$, exactly $P_*/q$
integers in $\{0,\ldots,P_*-1\}$ belong to the class
$a\bmod q$. Thus
\[
 \E(J_{\bm d}J_{\bm e})
 =\frac1{P_*}
   \#\{0\le t<P_*:t\equiv a\pmod q\}
 =\frac1q.
\]

To estimate the number of representatives in the shorter
interval, write
\[
 T=\ell q+h,\qquad 0\le h<q.
\]
Each of the $\ell$ complete blocks of $q$ consecutive
integers contains exactly one representative of $a\bmod q$,
and the remaining block contains at most one. Hence
\[
 \#\{1\le t\le T:t\equiv a\pmod q\}
 \in\{\ell,\ell+1\}.
\]
Since $T/q=\ell+h/q$, we obtain
\[
 \left|
 \#\{1\le t\le T:t\equiv a\pmod q\}-\frac Tq
 \right|\le1.
\]
Dividing by $T$ proves
\[
 \left|
 \frac1T\sum_{t=1}^T J_{\bm d}(t)J_{\bm e}(t)
 -\E(J_{\bm d}J_{\bm e})
 \right|\le\frac1T.
\]
This estimate therefore holds in both cases, for every pair
$\bm d,\bm e\in\mathcal D_k$, uniformly in $b$ and the shifts.

Applying the triangle inequality to the expansion of $R^2$ and
using~\eqref{eq:l1}, we obtain
\begin{equation}\label{eq:average-error}
 \begin{aligned}
 |S_1-\E R^2|
 &\le\frac1T\sum_{\bm d,\bm e}
                  |\lambda_{\bm d}|\,|\lambda_{\bm e}|\\
 &=\frac1T\left(\sum_{\bm d}|\lambda_{\bm d}|\right)^2
 \le\frac{D^3}{T}.
 \end{aligned}
\end{equation}
The bound is uniform in $b$ and in the shifts: the count of a
residue class depends on its modulus, but its error bound does
not depend on the position of the class.

\medskip\noindent\textit{Controlling the off-diagonal terms.}
We prove a second-moment estimate that will be applied twice.
For arbitrary real coefficients $c_{\bm r}$ on $\mathcal D_j$,
where $0\le j\le k$, put
\[
 F(t)=\sum_{\bm r\in\mathcal D_j}c_{\bm r}
                           \prod_{i=1}^j\psi_{i,r_i}(t),
 \qquad
 I=\sum_{\bm r\in\mathcal D_j}
                           \frac{c_{\bm r}^2}
                                {\prod_{i=1}^j\varphi(r_i)}.
\]
We claim that
\begin{equation}\label{eq:gram}
 |\E F^2-I|\ll\frac{k\log D}{z}\,I.
\end{equation}
For $j=0$, $F$ is a constant, so $\E F^2=I$ and there is nothing
to prove. Assume that $1\le j\le k$.

For $\bm r\in\mathcal D_j$, write $n(\bm r)=r_1\cdots r_j$.
The coordinates are pairwise coprime, so this product is
squarefree and
$\varphi(n(\bm r))=\prod_i\varphi(r_i)$. Introduce the normalized
functions and coefficients
\[
 \Psi_{\bm r}(t)
   =\sqrt{\varphi(n(\bm r))}\prod_{i=1}^j\psi_{i,r_i}(t),
 \qquad
 b_{\bm r}=\frac{c_{\bm r}}{\sqrt{\varphi(n(\bm r))}}.
\]
Then $F=\sum_{\bm r}b_{\bm r}\Psi_{\bm r}$ and
$I=\sum_{\bm r}b_{\bm r}^2$. By the independence of the residue
coordinates $t\bmod p$ for distinct primes $p\mid P_*$,
established above, and by \eqref{eq:local-moments}, we obtain
\[
 \E\Psi_{\bm r}^2
 =\varphi(n(\bm r))
    \prod_{p\mid n(\bm r)}\frac1{p-1}=1.
\]
Thus $I$ is exactly the diagonal part of the expansion of
$\E F^2$.

Let $M_{\bm r,\bm s}=\E(\Psi_{\bm r}\Psi_{\bm s})$ be the
matrix of normalized inner products. If $n(\bm r)\ne n(\bm s)$,
some prime divides precisely one of these two squarefree
products. At that prime the local average is $\E\psi_{i,p}=0$,
so independence implies $M_{\bm r,\bm s}=0$.
It remains to consider tuples with the same product $n$.

Such a tuple assigns each prime $p\mid n$ to one of the $j$
coordinates. For fixed $p$, the normalized local factors are
$\sqrt{p-1}\,\psi_{i,p}$. Their inner products are
\[
 \E\bigl((\sqrt{p-1}\,\psi_{i,p})
         (\sqrt{p-1}\,\psi_{\ell,p})\bigr)
 =\begin{cases}
    1,&i=\ell,\\
    -1/(p-1),&i\ne\ell,
  \end{cases}
\]
by~\eqref{eq:local-moments}. Fix a tuple $\bm r\in\mathcal D_j$ and put
$n=n(\bm r)$. For each prime $p\mid n$, let
$\iota_{\bm r}(p)$ denote the unique index $i$ such that
$p\mid r_i$. This index is unique because the coordinates
of $\bm r$ are pairwise coprime.

Now consider another tuple $\bm s\in\mathcal D_j$ with
$n(\bm s)=n$, and define $\iota_{\bm s}(p)$ similarly.
The normalized functions can be written as
\[
 \Psi_{\bm r}(t)
 =
 \prod_{p\mid n}
 \bigl(\sqrt{p-1}\,\psi_{\iota_{\bm r}(p),p}(t)\bigr),
\]
and likewise for $\Psi_{\bm s}$. By independence across
distinct prime moduli,
\[
 M_{\bm r,\bm s}
 =
 \prod_{p\mid n}
 \left(
 (p-1)\,
 \E\bigl(
 \psi_{\iota_{\bm r}(p),p}
 \psi_{\iota_{\bm s}(p),p}
 \bigr)
 \right).
\]
For each fixed $p\mid n$, the local factor is $1$ if
$\iota_{\bm s}(p)=\iota_{\bm r}(p)$, and it is
$-1/(p-1)$ otherwise. Define
\[
 h_p(\ell)=
 \begin{cases}
  1,&\ell=\iota_{\bm r}(p),\\[1mm]
  \dfrac1{p-1},&\ell\ne\iota_{\bm r}(p),
 \end{cases}
 \qquad 1\le\ell\le j.
\]
Taking absolute values in the preceding product gives
\[
 |M_{\bm r,\bm s}|
 =
 \prod_{p\mid n}h_p\bigl(\iota_{\bm s}(p)\bigr).
\]

We next describe precisely the tuples over which we sum.
For each prime $p\mid n$, choose an index
$f(p)\in\{1,\ldots,j\}$, and set
\[
 s_\ell=
 \prod_{\substack{p\mid n\\f(p)=\ell}}p,
 \qquad 1\le\ell\le j,
\]
where an empty product is understood to be $1$.
Every such choice produces a tuple in $\mathcal D_j$:
each $s_\ell$ divides $P_*$, the coordinates are pairwise
coprime, and their product is $n\le D$.
Conversely, every tuple $\bm s\in\mathcal D_j$ with
$n(\bm s)=n$ arises from exactly one such choice,
namely $f(p)=\iota_{\bm s}(p)$.

Thus the coordinate assigned to each prime can be chosen
freely, with no further restriction linking the choices
for different primes. Summing over all tuples with
product $n$ is therefore equivalent to summing over
all maps
\[
 f:\{p\in\PP:p\mid n\}\longrightarrow\{1,\ldots,j\}.
\]
By the distributive law for finite sums and products,
\[
 \begin{aligned}
 \sum_{\substack{\bm s\in\mathcal D_j\\n(\bm s)=n}}
       |M_{\bm r,\bm s}|
 &=
 \sum_{f:\{p\in\PP:p\mid n\}\to\{1,\ldots,j\}}
       \prod_{p\mid n}h_p(f(p))\\
 &=
 \prod_{p\mid n}
       \left(\sum_{\ell=1}^j h_p(\ell)\right)\\
 &=
 \prod_{p\mid n}
       \left(1+\frac{j-1}{p-1}\right).
 \end{aligned}
\]
Indeed, for each prime $p$, exactly one choice of $\ell$
agrees with $\iota_{\bm r}(p)$ and contributes $1$;
the remaining $j-1$ choices each contribute $1/(p-1)$.

This sum includes the diagonal entry
$|M_{\bm r,\bm r}|=1$, corresponding to the choice
$f(p)=\iota_{\bm r}(p)$ for every $p\mid n$.
Moreover, we have already shown that
$M_{\bm r,\bm s}=0$ whenever $n(\bm s)\ne n$.
Hence removing the diagonal entry gives the full
off-diagonal row sum:
\[
 \sum_{\substack{\bm s\in\mathcal D_j\\\bm s\ne\bm r}}
       |M_{\bm r,\bm s}|
 =
 \prod_{p\mid n}
       \left(1+\frac{j-1}{p-1}\right)-1.
\]
For $n=1$ this is $0$. Otherwise, $p>z$ and
$\omega(n)\le\log D/\log z$ (see \eqref{w(n).calc}) imply
\[
 \sum_{p\mid n}\frac{j-1}{p-1}
 \le\frac{2k\omega(n)}z
 \le\frac{2k\log D}{z\log z}
 \ll\frac{k\log D}{z}=O(x^{-4}).
\]
Here we used $k\le x$, $\log D=x/4$, and $z=x^6$.
Applying $1+h\le e^h$ and then $e^v-1\ll v$ for $0\le v\le1$,
we conclude that the absolute off-diagonal row sums satisfy
\[
 \rho:=\max_{\bm r\in\mathcal D_j}
       \sum_{\substack{\bm s\in\mathcal D_j\\\bm s\ne\bm r}}
       |M_{\bm r,\bm s}|
 \ll\frac{k\log D}{z}.
\]

Finally, $M$ is symmetric. Expanding $\E F^2$, removing its
diagonal part, and using $2|ab|\le a^2+b^2$, we obtain
\[
 \begin{aligned}
 |\E F^2-I|
 &\le\sum_{\bm r\ne\bm s}
         |M_{\bm r,\bm s}|\,|b_{\bm r}b_{\bm s}|\\
 &\le\frac12\sum_{\bm r\ne\bm s}
         |M_{\bm r,\bm s}|(b_{\bm r}^2+b_{\bm s}^2)\\
 &=\sum_{\bm r}b_{\bm r}^2
       \sum_{\bm s\ne\bm r}|M_{\bm r,\bm s}|\\
 &\le\rho\sum_{\bm r}b_{\bm r}^2
 \ll\frac{k\log D}{z}\,I.
 \end{aligned}
\]
This proves \eqref{eq:gram} for arbitrary real coefficients
$c_{\bm r}$. 

\medskip\noindent\textit{The diagonal mass of the weight.}
For the coefficients of $R$, put
\[
 I_R=\sum_{\bm r\in\mathcal D_k}
                 \prod_{i=1}^k\frac{a(r_i)^2}{\varphi(r_i)}.
\]
Our aim is to prove that
\[
 I_R=(1+o(1))A^k.
\]
We first compute the larger sum obtained by
removing both the pairwise coprimality condition and
the restriction $r_1\cdots r_k\le D$.

In this unrestricted sum, each coordinate $r_i$ ranges
freely over all divisors of $P_*$. Since the $i$th factor
depends only on $r_i$, the sum factors as
\[
 \begin{aligned}
 \sum_{r_1\mid P_*}\cdots\sum_{r_k\mid P_*}
       \prod_{i=1}^k\frac{a(r_i)^2}{\varphi(r_i)}
 &=
 \prod_{i=1}^k
       \left(
       \sum_{r_i\mid P_*}
       \frac{a(r_i)^2}{\varphi(r_i)}
       \right)\\
 &=A^k,
 \end{aligned}
\]
where the last equality follows from the definition of $A$.

The sum defining $I_R$ contains only the tuples belonging
to $\mathcal D_k$. Thus $I_R$ is obtained from the
unrestricted sum by discarding some nonnegative terms.
In particular,
\[
 \begin{aligned}
 0\le A^k-I_R
 &=
 \sum_{\substack{r_i\mid P_*\ (1\le i\le k)\\
                  (r_1,\ldots,r_k)\notin\mathcal D_k}}
       \prod_{i=1}^k\frac{a(r_i)^2}{\varphi(r_i)}.
 \end{aligned}
\]

A tuple is discarded if its product exceeds $D$, or if
two of its coordinates have a common prime divisor.
We estimate separately the total contribution of each
of these two types of excluded tuples.
A tuple violating both conditions may be counted twice;
this is harmless because we seek an upper bound and
all summands are nonnegative.

It therefore suffices to show that each of these two
contributions is $o(A^k)$. This will imply that
$A^k-I_R=o(A^k)$, and hence that $I_R/A^k\to1$.

First, for $r_1\cdots r_k>D$ and $\beta>0$, we have
$1\le(r_1\cdots r_k/D)^\beta$. Thus the mass beyond the cutoff
is at most
\[
 \begin{aligned}
 &\sum_{\substack{r_i\mid P_*\ (1\le i\le k)\\
                   r_1\cdots r_k>D}}
              \prod_{i=1}^k\frac{a(r_i)^2}{\varphi(r_i)}\\
 &\qquad\le D^{-\beta}
       \sum_{r_1\mid P_*}\cdots\sum_{r_k\mid P_*}
              \prod_{i=1}^k\frac{a(r_i)^2r_i^\beta}{\varphi(r_i)}
       =D^{-\beta}A_\beta^k.
 \end{aligned}
\]
Since $A\ge1$ and $A_\beta\le A+C_1\beta$, we have
\[
 A_\beta^k
 \le A^k\left(1+\frac{C_1\beta}{A}\right)^k
 \le A^k e^{C_1k\beta}.
\]
We now choose the constant in the lemma so that
\[
 0<\eta\le\min\left(1,\frac1{8C_1}\right).
\]
This ensures
$C_1k\le C_1\eta x\le x/8=\tfrac12\log D$.
Because $u=\beta\log D$, it follows that
\begin{equation}\label{eq:tail}
 D^{-\beta}A_\beta^k
 \le A^k\exp(-u+C_1k\beta)
 \le A^k e^{-u/2}.
\end{equation}

Second, we estimate the contribution of tuples whose
coordinates are not pairwise coprime. If $k=1$, there
are no pairs of distinct coordinates, so this
contribution is zero. We therefore assume that $k\ge2$.

If a tuple $(r_1,\ldots,r_k)$ fails the pairwise
coprimality condition, then $(r_i,r_j)>1$ for some
indices $i<j$. Hence there is a prime $p$ such that
\[
 p\mid r_i,\qquad p\mid r_j.
\]
Since every coordinate divides $P_*$, this prime
also divides $P_*$.

For each prime $p\mid P_*$, put
\[
 B_p=
 \sum_{\substack{d\mid P_*\\p\mid d}}
       \frac{a(d)^2}{\varphi(d)}.
\]
Fix a pair $i<j$ and a prime $p\mid P_*$. Consider all
tuples of divisors of $P_*$ satisfying
$p\mid r_i$ and $p\mid r_j$, without imposing either
the product cutoff or any coprimality conditions.
The coordinates can then be summed independently.
Each of the two specified coordinates contributes
$B_p$, whereas each of the remaining $k-2$ coordinates
contributes
\[
 \sum_{d\mid P_*}\frac{a(d)^2}{\varphi(d)}=A.
\]
Consequently,
\[
 \begin{aligned}
 &\sum_{\substack{r_\ell\mid P_*\ (1\le\ell\le k)\\
                  p\mid r_i,\ p\mid r_j}}
       \prod_{\ell=1}^k
       \frac{a(r_\ell)^2}{\varphi(r_\ell)}
 \\
 &\qquad=
 \left(
 \sum_{\substack{d\mid P_*\\p\mid d}}
       \frac{a(d)^2}{\varphi(d)}
 \right)^2
 \prod_{\substack{1\le\ell\le k\\\ell\ne i,j}}
 \left(
 \sum_{r_\ell\mid P_*}
       \frac{a(r_\ell)^2}{\varphi(r_\ell)}
 \right)
 \\
 &\qquad=B_p^2A^{k-2}.
 \end{aligned}
\]
For $k=2$, the product over the remaining coordinates
is empty and equals $1$, in agreement with $A^0=1$.

Every tuple failing pairwise coprimality is included
for at least one pair $(i,j)$ and at least one prime
$p\mid P_*$. A tuple may be included several times,
but all weights are nonnegative, so summing over
all pairs and primes gives an upper bound:
\[
 \begin{aligned}
 &\sum_{\substack{r_\ell\mid P_*\ (1\le\ell\le k)\\
          (r_i,r_j)>1\ \text{for some }i<j}}
       \prod_{\ell=1}^k
       \frac{a(r_\ell)^2}{\varphi(r_\ell)}
 \\
 &\qquad\le
 \sum_{1\le i<j\le k}\sum_{p\mid P_*}
       A^{k-2}B_p^2
 =
 \binom{k}{2}A^{k-2}\sum_{p\mid P_*}B_p^2.
 \end{aligned}
\]

It remains to bound $B_p$. By
\eqref{eq:basic-coefficients}, we have $|a(d)|\le1$,
and therefore $a(d)^2\le|a(d)|$. Applying
\eqref{eq:prime-coefficients}, we obtain
\[
 B_p
 \le
 \sum_{\substack{d\mid P_*\\p\mid d}}
       \frac{|a(d)|}{\varphi(d)}
 \ll\frac1p.
\]
All prime divisors of $P_*$ exceed $z$. Hence
\[
 \sum_{p\mid P_*}B_p^2
 \ll
 \sum_{p\mid P_*}\frac1{p^2}
 \le
 \sum_{\substack{n\in\mathbb N\\n>z}}\frac1{n^2}
 \ll\frac1z.
\]
For completeness, the last estimate follows from
\[
 \sum_{\substack{n\in\mathbb N\\n>z}}\frac1{n^2}
 \le
 \int_{\lfloor z\rfloor}^{\infty}\frac{dt}{t^2}
 =
 \frac1{\lfloor z\rfloor}
 \le\frac2z
 \qquad(z\ge2).
\]
Substituting these bounds and using $A\ge1$, we conclude
that
\[
 \sum_{\substack{r_\ell\mid P_*\ (1\le\ell\le k)\\
          (r_i,r_j)>1\ \text{for some }i<j}}
       \prod_{\ell=1}^k
       \frac{a(r_\ell)^2}{\varphi(r_\ell)}
 \ll
 \frac{k^2}{z}A^{k-2}
 \le
 \frac{k^2}{z}A^k.
\]
Thus, the total weight of all tuples with
$(r_i,r_j)>1$ for some $i<j$ is $O(k^2A^k/z)$.

Combining the tail bound~\eqref{eq:tail} with the preceding
coprimality bound, we have
\[
 0\le A^k-I_R
 \ll A^k\left(e^{-u/2}+\frac{k^2}{z}\right).
\]
Here $u=\tfrac14(\log x)^5\to\infty$ and
$k^2/z\le x^2/x^6=x^{-4}\to0$. Hence
\[
 I_R=A^k\left(1+O\!\left(e^{-u/2}+\frac{k^2}{z}\right)\right)
 =(1+o(1))A^k.
\]We now apply \eqref{eq:gram} to the function $R$.
In the general definition of $F$, take
\[
 j=k,
 \qquad
 c_{\bm r}=\prod_{i=1}^k a(r_i)
 \quad(\bm r\in\mathcal D_k).
\]
With this choice,
\[
 \begin{aligned}
 F(t)
 &=
 \sum_{\bm r\in\mathcal D_k}
 c_{\bm r}\prod_{i=1}^k\psi_{i,r_i}(t)\\
 &=
 \sum_{\bm r\in\mathcal D_k}
 \prod_{i=1}^k
       \bigl(a(r_i)\psi_{i,r_i}(t)\bigr)
 =R(t).
 \end{aligned}
\]
Moreover, the corresponding diagonal sum is precisely
$I_R$, since
\[
 \begin{aligned}
 I
 &=
 \sum_{\bm r\in\mathcal D_k}
 \frac{c_{\bm r}^{\,2}}
      {\prod_{i=1}^k\varphi(r_i)}\\
 &=
 \sum_{\bm r\in\mathcal D_k}
 \frac{\left(\prod_{i=1}^k a(r_i)\right)^2}
      {\prod_{i=1}^k\varphi(r_i)}\\
 &=
 \sum_{\bm r\in\mathcal D_k}
 \prod_{i=1}^k\frac{a(r_i)^2}{\varphi(r_i)}
 =I_R.
 \end{aligned}
\]
Thus \eqref{eq:gram}, with $F=R$ and $I=I_R$, gives
\[
 \left|\E R^2-I_R\right|
 \ll\frac{k\log D}{z}\,I_R.
\]
Since $I_R\le A^k$, it follows that
\[
 \E R^2
 =
 I_R+O\!\left(\frac{k\log D}{z}A^k\right).
\]

We have already proved that
\[
 I_R
 =
 A^k+
 O\!\left(
 A^k\left(e^{-u/2}+\frac{k^2}{z}\right)
 \right).
\]
Substituting this estimate into the preceding identity
and combining the error terms, we obtain
\[
 \begin{aligned}
 \E R^2
 &=
 A^k+
 O\!\left(
 A^k\left(
 e^{-u/2}+\frac{k^2}{z}+\frac{k\log D}{z}
 \right)
 \right)\\
 &=
 A^k\left(
 1+O\!\left(
 e^{-u/2}+\frac{k^2}{z}+\frac{k\log D}{z}
 \right)
 \right).
 \end{aligned}
\]
Finally, $T=\lfloor e^x\rfloor\ge e^x/2$ for large $x$, so
\[
 \frac{D^3}{T}\le2e^{3x/4-x}=2e^{-x/4}.
\]
Since $A^k\ge1$, the absolute error in~\eqref{eq:average-error}
is also a relative error of size at most $2e^{-x/4}$ after
division by $A^k$. Therefore
\[
 \frac{S_1}{A^k}
 =1+O\!\left(e^{-u/2}+\frac{k^2}{z}
                  +\frac{k\log D}{z}+e^{-x/4}\right)
 =1+o(1).
\]
All terms tend to zero uniformly for $k\le\eta x$. In particular,
\begin{equation}\label{eq:S1}
 S_1=(1+o(1))A^k>0
\end{equation}
for sufficiently large $x$.

\medskip\noindent\textit{Collapsing the coordinate of an uncovered form.}
We first estimate the contribution of $L_k$; the same argument
will apply to each other form after relabelling. For $v\mid P_*$
with $v\le D$, define
\[
 \varepsilon(v)=
   \sum_{\substack{d\mid P_*\\d\le D/v\\(d,v)=1}}
                    \frac{a(d)}{\varphi(d)},
\]
and put
\[
 R_k(t)=\sum_{\bm r\in\mathcal D_{k-1}}
       \varepsilon(r_1\cdots r_{k-1})
       \prod_{i=1}^{k-1}a(r_i)\psi_{i,r_i}(t).
\]
To explain this definition, fix the first $k-1$ coordinates of
a tuple in $\mathcal D_k$ and write $v=r_1\cdots r_{k-1}$.
The admissible last coordinates are exactly those divisors
$d\mid P_*$ satisfying $(d,v)=1$ and $d\le D/v$. Grouping the
terms of $R$ according to their first $k-1$ coordinates gives
the exact identity
\[
 R(t)=\sum_{\bm r\in\mathcal D_{k-1}}
       \left(\prod_{i=1}^{k-1}a(r_i)\psi_{i,r_i}(t)\right)
       \sum_{\substack{d\mid P_*\\d\le D/v\\(d,v)=1}}
                         a(d)\psi_{k,d}(t),
 \qquad v=r_1\cdots r_{k-1}.
\]
If $(L_k(t),P_*)=1$, then $\ind_{p\mid L_k(t)}=0$ for every
$p\mid P_*$, and consequently
\[
 \psi_{k,d}(t)
 =\prod_{p\mid d}\frac1{p-1}
 =\frac1{\varphi(d)}\qquad(d\mid P_*).
\]
The inner sum is then $\varepsilon(v)$, so $R(t)=R_k(t)$.
To justify the following inequality, consider the two
possible cases. If $(L_k(t),P_*)=1$, then $R(t)=R_k(t)$
by the preceding argument, and the indicator equals $1$.
Since $w(t)=R(t)^2$, we therefore have
\[
 w(t)\ind_{(L_k(t),P_*)=1}
 =R(t)^2
 =R_k(t)^2.
\]
If $(L_k(t),P_*)>1$, then the indicator equals $0$, so
\[
 w(t)\ind_{(L_k(t),P_*)=1}
 =0\le R_k(t)^2,
\]
where the inequality follows from the nonnegativity
of a square. Thus, in both cases, we obtain
\begin{equation}\label{eq:uncovered}
 w(t)\ind_{(L_k(t),P_*)=1}\le R_k(t)^2.
\end{equation}
When $k=1$, the set $\mathcal D_0$ contains only the empty tuple,
and these formulas mean $v=1$ and $R_1(t)=\varepsilon(1)$.

We also need to compare the two averages of $R_k^2$. Algebraically,
$R_k$ is obtained from $R$ by setting every prime divisibility
indicator for the $k$th form equal to zero. Indeed, this replaces
$\psi_{k,d}$ by $1/\varphi(d)$ in the product expansion.
In the expansion with coefficients $\lambda_{\bm d}$, this
operation removes exactly the terms with $d_k>1$, and leaves
those with $d_k=1$. Hence
\[
 R_k(t)=\sum_{(d_1,\ldots,d_{k-1})\in\mathcal D_{k-1}}
       \lambda_{(d_1,\ldots,d_{k-1},1)}
       \prod_{i=1}^{k-1}\ind_{d_i\mid L_i(t)}.
\]
In particular, by \eqref{eq:l1} its coefficient sum satisfies
\[
 \sum_{\bm d\in\mathcal D_{k-1}}
       |\lambda_{(\bm d,1)}|
 \le\sum_{\bm d\in\mathcal D_k}|\lambda_{\bm d}|
 \le D^{3/2}.
\]
Every coordinate product still lies below $D$. The residue-class
count used to prove~\eqref{eq:average-error} therefore applies
without change and gives
\begin{equation}\label{eq:collapsed-average}
 \left|\frac1T\sum_{t=1}^T R_k(t)^2-\E R_k^2\right|
 \le\frac1T\left(\sum_{\bm d\in\mathcal D_{k-1}}
                     |\lambda_{(\bm d,1)}|\right)^2
 \le\frac{D^3}{T}.
\end{equation}

\medskip\noindent\textit{Estimating the remaining coefficient sum.}
The full sum $\sum_{d\mid P_*}a(d)/\varphi(d)$ is zero by
\eqref{eq:basic-coefficients}. Since $v\le D$, the divisor $d=1$
occurs in the sum defining $\varepsilon(v)$. Every omitted divisor
is therefore greater than $1$, and its coefficient $a(d)$ is
negative. Subtracting the omitted terms from the zero full sum
gives the exact identity
\[
 \varepsilon(v)
 =-\sum_{\substack{d\mid P_*\\d>D/v\ \text{or }(d,v)>1}}
                 \frac{a(d)}{\varphi(d)}
 =\sum_{\substack{d\mid P_*\\d>D/v\ \text{or }(d,v)>1}}
                 \frac{|a(d)|}{\varphi(d)}\ge0.
\]
The condition $(d,v)>1$ means that some prime dividing $v$
also divides $d$. Taking the sum over these two possible reasons
for exclusion, with any multiple counting increasing the bound,
we obtain
\[
 \varepsilon(v)
 \le\sum_{\substack{d\mid P_*\\d>D/v}}
                    \frac{|a(d)|}{\varphi(d)}
    +\sum_{p\mid v}\sum_{\substack{d\mid P_*\\p\mid d}}
                    \frac{|a(d)|}{\varphi(d)}.
\]
For the first sum, $d>D/v$ implies $1\le(vd/D)^\beta$, so
\[
 \begin{aligned}
 \sum_{\substack{d\mid P_*\\d>D/v}}
          \frac{|a(d)|}{\varphi(d)}
 &\le\left(\frac vD\right)^\beta
       \sum_{\substack{d\mid P_*\\d>D/v}}
          \frac{|a(d)|d^\beta}{\varphi(d)}\\
 &\le\left(\frac vD\right)^\beta
       \sum_{\substack{d\mid P_*\\d>1}}
          \frac{|a(d)|d^\beta}{\varphi(d)}
 \ll\left(\frac vD\right)^\beta,
 \end{aligned}
\]
by~\eqref{eq:tilt}. For the second sum, use
\eqref{eq:prime-coefficients}, $p>z$, and
$\omega(v)\le\log v/\log z$ to get
\[
 \begin{aligned}
 \sum_{p\mid v}\sum_{\substack{d\mid P_*\\p\mid d}}
          \frac{|a(d)|}{\varphi(d)}
 &\ll\sum_{p\mid v}\frac1p
 \le\frac{\omega(v)}z\\
 &\le\frac{\log v}{z\log z}
 \le\frac{\log D}{z\log z}
 \le\frac{\log D}{z}.
 \end{aligned}
\]
For $v=1$, both prime sums are empty and these inequalities still
hold. Combining the estimates proves
\[
 0\le\varepsilon(v)
 \ll(v/D)^\beta+\frac{\log D}{z}.
\]
Squaring and applying $(a+b)^2\le2a^2+2b^2$, we also have
\begin{equation}\label{est:eps.2}
 \varepsilon(v)^2
 \ll D^{-2\beta}v^{2\beta}
      +\frac{(\log D)^2}{z^2}.
\end{equation}

\medskip\noindent\textit{The second moment of the collapsed weight.}
The diagonal sum for $R_k$ is
\[
 I_k=\sum_{\bm r\in\mathcal D_{k-1}}
       \varepsilon(v)^2
       \prod_{i=1}^{k-1}\frac{a(r_i)^2}{\varphi(r_i)},
 \qquad v=r_1\cdots r_{k-1}.
\]
For every tuple $\bm r\in\mathcal D_{k-1}$, the product
$v=r_1\cdots r_{k-1}$ satisfies $v\mid P_*$ and $v\le D$.
By \eqref{est:eps.2}, we therefore have
\[
 \varepsilon(v)^2
 \le C\left(
 D^{-2\beta}v^{2\beta}
 +\frac{(\log D)^2}{z^2}
 \right)
\]
with an absolute constant $C>0$, uniformly over these
tuples. Moreover,
\[
 v^{2\beta}
 =(r_1\cdots r_{k-1})^{2\beta}
 =\prod_{i=1}^{k-1}r_i^{2\beta}.
\]
Multiplying the bound for $\varepsilon(v)^2$ by the
nonnegative weight
$\prod_{i=1}^{k-1}a(r_i)^2/\varphi(r_i)$
and summing over $\mathcal D_{k-1}$, we obtain
\[
 \begin{aligned}
 I_k
 &\le
 CD^{-2\beta}
 \sum_{\bm r\in\mathcal D_{k-1}}
       \prod_{i=1}^{k-1}
       \frac{a(r_i)^2r_i^{2\beta}}{\varphi(r_i)}
 \\
 &\quad+
 C\frac{(\log D)^2}{z^2}
 \sum_{\bm r\in\mathcal D_{k-1}}
       \prod_{i=1}^{k-1}
       \frac{a(r_i)^2}{\varphi(r_i)}.
 \end{aligned}
\]

We now enlarge the domain of each sum. Membership in
$\mathcal D_{k-1}$ imposes the conditions
\[
 r_i\mid P_*,
 \qquad
 (r_i,r_j)=1\quad(i\ne j),
 \qquad
 r_1\cdots r_{k-1}\le D.
\]
We retain only the conditions $r_i\mid P_*$ and remove
the coprimality and product restrictions.
Since every summand is nonnegative, this can only
increase the sums.

For $\gamma\in\{0,2\beta\}$, it follows that
\[
 \begin{aligned}
 &\sum_{\bm r\in\mathcal D_{k-1}}
       \prod_{i=1}^{k-1}
       \frac{a(r_i)^2r_i^\gamma}{\varphi(r_i)}
 \\
 &\quad\le
 \sum_{r_1\mid P_*}\cdots\sum_{r_{k-1}\mid P_*}
       \prod_{i=1}^{k-1}
       \frac{a(r_i)^2r_i^\gamma}{\varphi(r_i)}
 \\
 &\quad=
 \prod_{i=1}^{k-1}
 \left(
 \sum_{r_i\mid P_*}
       \frac{a(r_i)^2r_i^\gamma}{\varphi(r_i)}
 \right)
 =A_\gamma^{k-1},
 \end{aligned}
\]
by the definition of $A_\gamma$.

Taking $\gamma=2\beta$ in the first sum and $\gamma=0$
in the second, and recalling that $A_0=A$, we conclude
that
\[
 I_k
 \ll
 D^{-2\beta}A_{2\beta}^{k-1}
 +\frac{(\log D)^2}{z^2}A^{k-1}.
\]

For completeness, consider the case $k=1$ separately.
Then $\mathcal D_{k-1}=\mathcal D_0$ contains exactly one
element, namely the empty tuple. Its coordinate product
is $v=1$, and the empty product of weights also equals $1$.
Thus the sum defining $I_1$ has a single term, and
\[
 I_1=\varepsilon(1)^2.
\]
The estimate \eqref{est:eps.2} is valid at $v=1$, since $1\mid P_*$ and $1\le D$.
Substituting $v=1$ therefore gives
\[
 I_1=\varepsilon(1)^2
 \ll D^{-2\beta}+\frac{(\log D)^2}{z^2}.
\]
This is precisely the claimed bound for $I_k$ when
$k=1$, because
\[
 A_{2\beta}^{k-1}=A_{2\beta}^{0}=1,
 \qquad
 A^{k-1}=A^0=1.
\]
Hence the estimate for $I_k$ holds also in this case.

 Applying~\eqref{eq:gram} with $j=k-1$ to the coefficients
$\varepsilon(v)\prod_{i<k}a(r_i)$ yields
\[
 \E R_k^2
 \le\left(1+C_2\frac{k\log D}{z}\right)I_k
 \ll D^{-2\beta}A_{2\beta}^{k-1}
          +\frac{(\log D)^2}{z^2}A^{k-1}.
\]
Here $C_2$ is an absolute constant, and the multiplier is bounded
because $k\log D/z=o(1)$.

We now estimate the first term.
Taking $\gamma=2\beta$ in \eqref{eq:tilt}, we obtain
\[
 A_{2\beta}\le A+2C_1\beta.
\]
Since $k\ge1$ and $A\ge1$, it follows that
\[
 \begin{aligned}
 A_{2\beta}^{k-1}
 &\le (A+2C_1\beta)^{k-1}\\
 &=A^{k-1}
   \left(1+\frac{2C_1\beta}{A}\right)^{k-1}\\
 &\le A^{k-1}
   \exp\!\left(\frac{2C_1(k-1)\beta}{A}\right)\\
 &\le A^{k-1}
   \exp\!\left(2C_1(k-1)\beta\right).
 \end{aligned}
\]

By the definition $u=\beta\log D$, we have
\[
 D^{-2\beta}
 =\exp(-2\beta\log D)
 =e^{-2u}.
\]
Multiplying the preceding estimate by this factor gives
\[
 \begin{aligned}
 D^{-2\beta}A_{2\beta}^{k-1}
 &\le
 A^{k-1}\exp\!\left(-2u+2C_1(k-1)\beta\right)\\
 &\le
 A^k\exp\!\left(-2u+2C_1k\beta\right).
 \end{aligned}
\]

It remains to control the positive term in the exponent.
The constant $\eta$ was chosen so that
$\eta\le1/(8C_1)$. Since $k\le\eta x$ and
$\log D=x/4$, we have
\[
 2C_1k
 \le2C_1\eta x
 \le\frac{x}{4}
 =\log D.
\]
Multiplying by $\beta>0$ yields
\[
 2C_1k\beta
 \le\beta\log D
 =u.
\]
Therefore
\[
 -2u+2C_1k\beta\le-u,
\]
and consequently
\[
 D^{-2\beta}A_{2\beta}^{k-1}
 \le A^k e^{-u}.
\]
The second term is at most $(\log D)^2A^k/z^2$, again because
$A\ge1$. We have therefore proved
\[
 \E R_k^2
 \ll A^k\left(e^{-u}+\frac{(\log D)^2}{z^2}\right).
\]

\medskip\noindent\textit{Bounding the number of uncovered forms.}
Combining~\eqref{eq:uncovered} and~\eqref{eq:collapsed-average},
we find
\[
 \frac1T\sum_{t=1}^T
       w(t)\ind_{(L_k(t),P_*)=1}
 \le\E R_k^2+\frac{D^3}{T}
 \ll A^k\left(e^{-u}+\frac{(\log D)^2}{z^2}\right)
       +\frac{D^3}{T}.
\]
The tuple set, the coefficient bounds, and the local moment
estimates are invariant under a permutation of the forms.
Repeating the same argument with each $L_i$ as the last form
and adding the $k$ estimates gives
\[
 \begin{aligned}
 S_2
 &=\sum_{i=1}^k\frac1T\sum_{t=1}^T
                    w(t)\ind_{(L_i(t),P_*)=1}\\
 &\ll kA^k\left(e^{-u}+\frac{(\log D)^2}{z^2}\right)
           +\frac{kD^3}{T}.
 \end{aligned}
\]
By~\eqref{eq:S1}, $S_1\ge A^k/2$ for sufficiently large $x$.
Dividing by this lower bound, using $A^k\ge1$ and the parameter
choices in~\eqref{eq:parameters}, which give
$u=(\log x)^5/4$, $\log D=x/4$, $z=x^6$, and
$D^3/T\le2e^{-x/4}$, we obtain
\[
 \begin{aligned}
 \frac{S_2}{S_1}
 &\ll k e^{-u}+\frac{k(\log D)^2}{z^2}
                    +\frac{kD^3}{TA^k}\\
 &\ll k e^{-(\log x)^5/4}
          +\frac{kx^2}{x^{12}}+k e^{-x/4}\\
 &\ll x e^{-(\log x)^5/4}+x^{-9}+x e^{-x/4}
 =o(1).
 \end{aligned}
\]
Thus $S_2<S_1$ for sufficiently large $x$.

If every $1\le t\le T$ had an uncovered form, its number of
uncovered forms would be at least $1$, and the nonnegative
weights would give $S_2\ge S_1$, a contradiction. Hence some
$t\in\{1,\ldots,T\}$ satisfies $(L_i(t),P_*)>1$ for every $i$.
Each such greatest common divisor has a prime factor, and every
prime factor of $P_*$ lies in $(z,y]$. Since $T\le e^x$, this is
exactly the divisibility statement of the lemma.
\end{proof}
\section{Proof of the main result}

\begin{proof}[Proof of Theorem~\ref{thm:main}]
Let $X$ be sufficiently large and put
\[
 x=\frac13\log X,\qquad Q=\prod_{p\le x}p,\qquad
 H=\lfloor c_0\log x\rfloor,
\]
where the absolute constant $c_0>0$ will be chosen below.
Then $X=e^{3x}$, and $H\ge1$ once $X$ is sufficiently large.
We shall find an integer $N$ for which all the $H$ integers
$N+1,\ldots,N+H$ lie in $[1,X]$ and have no Romanoff
representation.

\medskip\noindent\textit{Choosing the relevant exponents.}
Define
\[
 m=\left\lfloor\frac{\log(X/2)}{\log2}\right\rfloor,
 \qquad U=2^m.
\]
The defining inequalities for the floor function give
\[
 m\le\frac{\log(X/2)}{\log2}<m+1.
\]
Exponentiating to base $2$, we obtain
$U\le X/2<2U$, or equivalently
\[
 \frac X4<U\le\frac X2.
\]
We will place our block between $U$ and $2U$. This ensures
that only the powers $2^n$ with $n\le m$ need to be excluded
by divisibility arguments.

\medskip\noindent\textit{Averaging over the small primes.}
Consider the set of distinct shifts
\[
 S_0=\{j-2^n:1\le j\le H,\ 1\le n\le m\}.
\]
There are $Hm$ pairs $(j,n)$, so $|S_0|\le Hm$; different
pairs yielding the same shift contribute only one element.
For a fixed $s\in S_0$, translation by $s$ permutes the residue
classes modulo $Q$. Exactly $\varphi(Q)$ of these classes are
coprime to $Q$. Interchanging two finite sums therefore gives
\[
 \begin{aligned}
 \frac1Q\sum_{b_0\bmod Q}
       \#\{s\in S_0:(b_0+s,Q)=1\}
 &=\frac1Q\sum_{b_0\bmod Q}
       \sum_{s\in S_0}\ind_{(b_0+s,Q)=1}\\
 &=\frac1Q\sum_{s\in S_0}
       \sum_{b_0\bmod Q}\ind_{(b_0+s,Q)=1}\\
 &=|S_0|\frac{\varphi(Q)}Q.
 \end{aligned}
\]
At least one term in a finite average is no larger than the
average itself. Hence there exists a class $b_0\bmod Q$ such
that the set
\[
 S=\{s\in S_0:(b_0+s,Q)=1\}
\]
satisfies $|S|\le|S_0|\varphi(Q)/Q\le Hm\varphi(Q)/Q$.

We now evaluate this bound. The definition of $m$ gives
\[
 m=\frac{\log(X/2)}{\log2}+O(1)
   =\frac{3x}{\log2}+O(1),
 \qquad H=c_0\log x+O(1).
\]
Since $Q$ is squarefree, Mertens' product formula yields
\[
 \frac{\varphi(Q)}Q
 =\prod_{p\le x}\left(1-\frac1p\right)
 =\frac{e^{-\gamma_E}+o(1)}{\log x}.
\]
Multiplying the three estimates, with $c_0$ fixed, gives
\[
 Hm\frac{\varphi(Q)}Q
 =(c_0\log x+O(1))
   \left(\frac{3x}{\log2}+O(1)\right)
   \frac{e^{-\gamma_E}+o(1)}{\log x}
 =\left(\frac{3e^{-\gamma_E}c_0}{\log2}+o(1)\right)x.
\]
Consequently,
\begin{equation}\label{eq:survivors}
 |S|\le Hm\frac{\varphi(Q)}Q
 \le\left(\frac{3e^{-\gamma_E}}{\log2}c_0+o(1)\right)x.
\end{equation}
Choose $c_0>0$ so small that
$3e^{-\gamma_E}c_0/\log2<\eta/2$, where $\eta$ is the absolute
constant in Lemma~\ref{lem:translate}. The $o(1)$ term has
absolute value less than $\eta/2$ for sufficiently large $X$,
so~\eqref{eq:survivors} implies $|S|\le\eta x$.

The remaining hypotheses on the shifts also hold. They are
distinct because $S$ is a set, and for $s=j-2^n\in S$,
\[
 |s|\le j+2^n\le H+U\le H+X/2\le X=e^{3x}
\]
for sufficiently large $X$, since $H=O(\log\log X)=o(X)$.

\medskip\noindent\textit{Choosing the position of the block.}
Let $r$ be the representative of $b_0$ in $\{0,\ldots,Q-1\}$.
Set
\[
 b=r+Q\left\lceil\frac{3U/2-r}{Q}\right\rceil,
\]
where $\lceil v\rceil$ is the least integer not smaller than $v$.
Then $b\equiv b_0\pmod Q$. The inequalities
$v\le\lceil v\rceil<v+1$ give
\begin{equation}\label{eq:baseline}
 \frac{3U}{2}\le b<\frac{3U}{2}+Q.
\end{equation}
Apply Lemma~\ref{lem:translate} to $S$ and this integer $b$.
It supplies an integer $t$ with $1\le t\le e^x$ such that
every $b+Qt+s$, $s\in S$, has a prime divisor in $(z,y]$.
This application is also valid when $S$ is empty, by the
$k=0$ case of the lemma. Put $N=b+Qt$.

The prime number theorem, in the form
$\sum_{p\le x}\log p=x+o(x)$, gives
\[
 \log Q=\sum_{p\le x}\log p=(1+o(1))x.
\]
Therefore
\[
 Qe^x=\exp((2+o(1))x)
     =X^{2/3+o(1)}=o(X).
\]
For the last equality, the exponent $2/3+o(1)$ is eventually
less than, for example, $5/6$. Since $Q\le Qe^x$ and
$H=o(X)$, it follows that $Q+Qe^x+H=o(X)$. Also $U>X/4$, so
\[
 \frac{Q+Qe^x+H}{U}
 <\frac{4(Q+Qe^x+H)}X=o(1).
\]
In particular, $Q+Qe^x+H<U/4$ for large $X$.
The lower bound in~\eqref{eq:baseline} and $t\ge1$ give
$N\ge3U/2+Q>3U/2$. The upper bound and $t\le e^x$ give
\[
 N+H<\frac{3U}{2}+Q+Qe^x+H
      <\frac{3U}{2}+\frac U4=\frac{7U}{4}.
\]
Together with $2U\le X$, these inequalities prove
\begin{equation}\label{eq:location}
 \frac{3U}{2}<N,
 \qquad N+H<\frac{7U}{4}<2U\le X.
\end{equation}
Thus the entire block lies in $[1,X]$.

\medskip\noindent\textit{Excluding the exponents $n\le m$.}
Fix $1\le j\le H$ and $1\le n\le m$, and put $s=j-2^n\in S_0$.
If $s\notin S$, then $(b_0+s,Q)>1$, so some prime $p\mid Q$
divides $b_0+s$. Because $N=b+Qt\equiv b_0\pmod Q$,
the same prime divides
$N+s=N+j-2^n$. Every prime factor of $Q$ is at most $x$.
If $s\in S$, the lemma instead provides a prime divisor
$p\in(z,y]$ of $N+s$. Thus in either case the number has a
prime divisor at most $y$. Here $y>x$ for large $x$, since
\[
 \frac{\log y}{\log x}
 =\frac{x}{(\log x)^6}\longrightarrow\infty.
\]

To verify that the divisor is proper, note that $2^n\le U$ and
$N>3U/2$, so
\[
 N+j-2^n>\frac{3U}{2}-U=\frac U2.
\]
Moreover,
\[
 \frac{\log y}{\log X}
 =\frac1{3(\log x)^5}\longrightarrow0,
\]
which means $y=X^{o(1)}$. In particular, $y\le X^{1/2}<X/8$
for sufficiently large $X$. Since $U>X/4$, we conclude that
\[
 N+j-2^n>\frac U2>\frac X8>y.
\]
The integer $N+j-2^n$ is therefore positive and exceeds its
prime divisor, so it is composite. It cannot equal the prime
summand in a Romanoff representation with $n\le m$.

\medskip\noindent\textit{Excluding the remaining exponents.}
If $n\ge m+1$, then
\[
 2^n\ge2^{m+1}=2U>N+H\ge N+j
\]
by~\eqref{eq:location}. Hence $N+j-2^n<0$ and cannot be a
positive prime. We have ruled out all $n\ge1$, proving
\begin{equation}\label{FINAL}
 \{N+1,\ldots,N+H\}\cap\cR=\varnothing.
\end{equation}

\medskip\noindent\textit{The length of the block.} From \eqref{FINAL} we obtain $G_{\mathcal{R}}(X)\ge H$.
Since $x=(\log X)/3$, we have
\[
 H=\lfloor c_0\log x\rfloor
 \ge\left\lfloor\frac{c_0}{2}\log\log X\right\rfloor
\]and hence $G_{\mathcal{R}}(X)\gg \log\log X$, as stated in~\eqref{eq:main}. This completes the proof of Theorem \ref{thm:main}. 
\end{proof}

\section{Other bases}

\begin{proof}[Proof of Corollary~\ref{cor:base}]
Fix an integer $a\ge2$. Retain $x=(\log X)/3$ and
$Q=\prod_{p\le x}p$, and define
\[
 m=\left\lfloor\frac{\log(X/a)}{\log a}\right\rfloor,
 \qquad U=a^m,
 \qquad H=\lfloor c_a\log x\rfloor,
\]
where $c_a>0$ is fixed sufficiently small in terms of $a$.
As before, the floor inequalities imply
\[
 a^m\le X/a<a^{m+1},
 \qquad\text{and hence}\qquad
 \frac X{a^2}<U\le\frac Xa.
\]
For sufficiently large $X$, both $m$ and $H$ are positive.

Let
\[
 S_0=\{j-a^n:1\le j\le H,\ 1\le n\le m\},
\]
with repeated shifts counted only once. Thus $|S_0|\le Hm$.
For every fixed shift, translation permutes the residue
classes modulo $Q$, so the same finite averaging identity gives
\[
 \frac1Q\sum_{b_0\bmod Q}
       \#\{s\in S_0:(b_0+s,Q)=1\}
 =|S_0|\frac{\varphi(Q)}Q.
\]
Choosing a class whose count is at most this average and putting
$S=\{s\in S_0:(b_0+s,Q)=1\}$, we obtain
$|S|\le Hm\varphi(Q)/Q$.
Here
\[
 m=\frac{3x}{\log a}+O_a(1),\qquad
 H=c_a\log x+O(1),\qquad
 \frac{\varphi(Q)}Q
   =\frac{e^{-\gamma_E}+o(1)}{\log x}.
\]
Multiplying these expressions, just as in the proof of
\eqref{eq:survivors}, yields
\[
 |S|\le Hm\frac{\varphi(Q)}Q
 \le\left(\frac{3e^{-\gamma_E}c_a}{\log a}+o_a(1)\right)x.
\]
Choose $c_a$ so that $3e^{-\gamma_E}c_a/\log a<\eta/2$.
Then $|S|\le\eta x$ for sufficiently large $X$, with a threshold
allowed to depend on $a$. Moreover, every $s=j-a^n\in S$ satisfies
\[
 |s|\le H+U\le H+X/a\le X=e^{3x}
\]
for large $X$, since $a\ge2$ and $H=o(X)$. Thus all the shift
hypotheses of Lemma~\ref{lem:translate} hold.

Let $r\in\{0,\ldots,Q-1\}$ represent $b_0$, and take
\[
 b=r+Q\left\lceil\frac{(a+1)U/2-r}{Q}\right\rceil.
\]
The ceiling inequalities show that $b\equiv b_0\pmod Q$ and
\[
 \frac{a+1}{2}U\le b<\frac{a+1}{2}U+Q.
\]
Apply the lemma to $S$ and $b$, and set $N=b+Qt$ for the
resulting integer $1\le t\le e^x$. The estimate
$Qe^x=o(X)$ proved above is unchanged. For fixed $a$, the bound
$U>X/a^2$ implies
\[
 \frac{Q+Qe^x+H}{U}
 <\frac{a^2(Q+Qe^x+H)}X=o_a(1).
\]
Since $(a-1)/2>0$, we may therefore suppose that
$Q+Qe^x+H<(a-1)U/2$. Using $t\ge1$ for the lower bound
and $t\le e^x$ for the upper bound, we find
\[
 N>\frac{a+1}{2}U,
 \qquad
 N+H<\frac{a+1}{2}U+Q+Qe^x+H<aU\le X.
\]

Fix $1\le j\le H$ and $1\le n\le m$. If $j-a^n\notin S$,
then the choice of $b_0$ supplies a prime $p\mid Q$ dividing
$N+j-a^n$. If $j-a^n\in S$, the lemma supplies a prime in
$(z,y]$ dividing the same integer. Thus in both cases there is
a prime divisor at most $y$, and
\[
 N+j-a^n>\frac{a+1}{2}U-U=\frac{a-1}{2}U.
\]
This lower bound exceeds $y$ for large $X$: indeed,
\[
 \frac{a-1}{2}U>\frac{a-1}{2a^2}X,
 \qquad \frac yX=X^{-1+o(1)}\longrightarrow0,
\]
and $(a-1)/(2a^2)$ is a fixed positive constant. The prime
divisor is therefore proper, and $N+j-a^n$ is composite.
For $n\ge m+1$, instead,
\[
 a^n\ge a^{m+1}=aU>N+H\ge N+j,
\]
so $N+j-a^n<0$ and cannot be a positive prime.
All exponents are excluded, and hence
\[
 \{N+1,\ldots,N+H\}\cap\cR_a=\varnothing.
\]
Finally, $\log x\ge\tfrac12\log\log X$ for large $X$, so
\[
 G_{\cR_a}(X)\ge H
 \ge\left\lfloor\frac{c_a}{2}\log\log X\right\rfloor
 \gg_a\log\log X.
\]
This proves the corollary.
\end{proof}

\section{Acknowledgements}
This work is an output of a research project (HSE-BR-2025-024) implemented as part of the Basic Research Program at HSE University.


\begin{thebibliography}{99}

\bibitem{Erdos}
P.~Erd\H{o}s.
\emph{On integers of the form $2^k+p$ and some related problems}.
Summa Brasiliensis Mathematicae \textbf{2} (1950), 113--123.

\bibitem{EulerGoldbach}
L.~Euler.
\emph{Letter to Christian Goldbach, 16 December 1752}.
In: F.~Lemmermeyer and M.~Mattm\"uller (eds.),
\emph{Correspondence of Leonhard Euler with Christian Goldbach}.
Publikationen des Bernoulli-Euler-Zentrums, vol.~1,
Bernoulli-Euler-Gesellschaft, Basel, 2016.
Opera Omnia, Series~IVA, vol.~4, online edition.
\url{https://doi.org/10.12685/publbez.1.2016}.

\bibitem{HabsiegerRoblot}
L.~Habsieger and X.-F.~Roblot.
\emph{On integers of the form $p+2^k$}.
Acta Arithmetica \textbf{122} (2006), no.~1, 45--50.


\bibitem{KK}
A.~B.~Kalmynin and S.~V.~Konyagin.
\emph{Large gaps between Romanov numbers} (in Russian).
Chebyshevskii Sbornik \textbf{27} (2026), no.~2, 180--186.


\bibitem{LG}
OpenAI.
\emph{Improved Long Gaps Between Primes}.
Preprint supplied as \texttt{long\_gaps.pdf}, 8~pp.
\url{https://github.com/openai/LongGapsBetweenPrimes}.

\bibitem{Maynard}
J.~Maynard.
\emph{Small gaps between primes}.
Ann. of Math. (2) \textbf{181} (2015), no.~1, 383--413.

\bibitem{Polignac}
A.~de Polignac.
\emph{Recherches nouvelles sur les nombres premiers}.
Comptes rendus de l'Acad\'emie des sciences \textbf{29} (1849),
397--401.

\bibitem{Radomskii}
A.~Radomskii.
\emph{Variants of Romanoff's theorem}.
Preprint, arXiv:2504.09954, 2025; version~7, 2026.
\url{https://arxiv.org/abs/2504.09954v7}.

\bibitem{Rieger}
G.~J.~Rieger.
\emph{Verallgemeinerung zweier S\"atze von Romanov aus der additiven
Zahlentheorie}.
Mathematische Annalen \textbf{144} (1961), 49--55.


\bibitem{Romanoff}
N.~P.~Romanoff.
\emph{\"Uber einige S\"atze der additiven Zahlentheorie}.
Mathematische Annalen \textbf{109} (1934), 668--678.



\bibitem{ShparlinskiWeingartner}
I.~E.~Shparlinski and A.~J.~Weingartner.
\emph{An explicit polynomial analogue of Romanoff's theorem}.
Finite Fields and Their Applications \textbf{44} (2017), 22--33.

\end{thebibliography}
\end{document}